\documentclass[10pt,reqno,a4paper]{amsart}%
\usepackage{amsfonts,amsmath}%,times
\usepackage{tikz}
\usepackage{shuffle}

\usepackage[sc]{mathpazo} % Palatino with smallcaps
\usepackage[scaled]{helvet} % Helvetica, scaled 95%
\usepackage{eulervm} % Euler math

\definecolor{grau}{rgb}{0.65,0.65,0.65}
\definecolor{dblau}{rgb}{0,0,0.45}
\definecolor{blau}{rgb}{0,0,0.75} %colour for in-document links
\definecolor{grun}{rgb}{0.1,0.6,0.1} %colour for in-document links
\usepackage[pdftex]{hyperref}
\hypersetup{colorlinks,linkcolor=grun,citecolor=blue,urlcolor=blau}

\newcommand{\myt}[1]{\emph{\color{dblau}#1}}

\theoremstyle{plain}
\newtheorem{lem}{\normalfont\scshape Lemma}
\newtheorem{thm}{\normalfont\scshape Theorem}
\newtheorem{coroll}{\normalfont\scshape Corollary}
\newtheorem{prop}{Proposition}

\theoremstyle{definition}
\newtheorem{remark}{\normalfont\scshape Remark}
\newtheorem{example}{\normalfont\scshape Example}
\newtheorem{defi}{\normalfont\scshape Definition}

\def\ztt{\zeta^{t}}
\def\zt{\zeta}
\def\zts{\zeta^{\star}}

\newcommand{\Ha}[2]{\ensuremath{H_{#1}^{(#2)}}}

\renewcommand{\P}{\ensuremath{\mathbb{P}}}
\newcommand{\law}{\ensuremath{\stackrel{(d)}=}}
\newcommand{\claw}{\ensuremath{\xrightarrow{(d)}}}

\newcommand{\laenge}{\ensuremath{\mathcal{L}}}

\newcommand{\M}{\ensuremath{\mathcal{M}}}
\renewcommand{\S}{\ensuremath{\mathcal{R}}}

\newcommand{\N}{\ensuremath{\mathbb{N}}}
\newcommand{\E}{\ensuremath{\mathbb{E}}}
\newcommand{\V}{\ensuremath{\mathbb{V}}}

\newcommand{\fallfak}[2]{\ensuremath{#1^{\underline{#2}}}}
\newcommand{\stir}[2]{\genfrac{ [ }{ ] }{0pt}{}{#1}{#2}}
\newcommand{\Stir}[2]{\genfrac{ \{ }{ \} }{0pt}{}{#1}{#2}}

\DeclareMathOperator{\Exp}{\text{Exp}}
\DeclareMathOperator{\Hy}{\text{Hypergeo}}
\DeclareMathOperator{\Po}{\text{Poisson}}
\DeclareMathOperator{\Be}{\text{Be}}
\DeclareMathOperator{\Beta}{\text{B}}

\begin{document}
\title{On multisets, interpolated multiple zeta values and limit laws}%
\author{Markus Kuba}
\date{\today}

\begin{abstract}
In this work we discuss a parameter $\sigma$ on weighted $k$-element multisets of $[n]= \{1,\dots ,n\}$. 
The sums of weighted $k$-multisets are related to $k$-subsets, $k$-multisets, as well as special instances of truncated interpolated multiple zeta values.
We study properties of this parameter using symbolic combinatorics. We (re)derive and extend certain identities for $\ztt_n(\{m\}_k)$.
Moreover, we introduce random variables on the $k$-element multisets and derive their distributions, as well as limit laws for $k$ or $n$ tending to infinity. 
\end{abstract}


\maketitle

%\emph{Keywords:} $k$-elementary multiset, truncated multiple zeta values, interpolated multiple zeta star values, Harmonic numbers, limit %laws. \\
%\indent\emph{2010 Mathematics Subject Classification} 11M32, 60C05.

\section{Introduction}
The \myt{multiple zeta values} are defined as
\[
\zt(i_1,\dots,i_k)=\sum_{\ell_1>\cdots>\ell_k\ge 1}\frac1{\ell_1^{i_1}\cdots \ell_k^{i_k}},
\]
with admissible indices $(i_1,\dots,i_k)$ satisfying $i_1\ge 2$, $i_j\ge 1$ for $2\le j\le k$, see Hoffman~\cite{H92} or Zagier~\cite{Z}. 
Their \myt{truncated} counterpart, sometimes called \myt{multiple harmonic sum}, is given by
\[
\zt_n(i_1,\dots,i_k)=\sum_{n\ge \ell_1>\cdots>\ell_k\ge 1}\frac1{\ell_1^{i_1}\cdots \ell_k^{i_k}}.
\]
We refer to $i_1 +\dots + i_k$ as the weight of this multiple zeta value, and $k$ as its depth. 
For a comprehensive overview as well as a great many pointers to the literature we refer to the survey of Zudilin~\cite{Zu}.
An important variant of the (truncated) multiple zeta values are the so-called multiple zeta \myt{star values}, where equality is allowed:
\[
\begin{split}
  \zts(i_1,\dots,i_k) & = \sum_{\ell_1\ge \cdots\ge \ell_k\ge 1}\frac1{\ell_1^{i_1}\cdots \ell_k^{i_k}} %\quad \text{and}\\
  %\zts_{n}(i_1,\dots,i_k) & =\sum_{n \ge \ell_1\ge \cdots\ge \ell_k\ge 1}\frac1{\ell_1^{i_1}\cdots \ell_k^{i_k}}
	,
\end{split}
\]
their truncated counterpart $\zts_{n}(i_1,\dots,i_k)$ is defined analogous. On the combinatorial side, it is well known that two special cases,
namely the series $\zt_{n}(\{1\}_k)$ and $\zts_{n}(\{1\}_k)$ occur in a multitude of different places. 
See for example~\cite{KP} and the references therein for different representation of the two series.
The value $\zt_{n}(\{1\}_k)$ is closely related to the Stirling numbers of the first kind, and $\zts_{n}(\{1\}_k)$ is related to alternating binomial sums.
Here and throughout this work $\{m\}_{k}$ means $m$ repeated $k$ times. 
For non-truncated series $\zt$ and $\zts$, Yamamoto~\cite{Y} introduced a generalization of both versions called \myt{interpolated multiple zeta values}. Noting that, 
\[
\zts(i_1,\dots,i_k)=\sum_{\circ = \text{``},\text{''} \text{or} \, \text{``}+\text{''}}\zt(i_1\circ i_2 \dots \circ i_k),
\]
let the parameter $\sigma$ denote the number of plus in the expression $i_1\circ i_2 \dots \circ i_k$. Yamamoto defines
\begin{equation}
\label{EqYamamoto}
  \ztt(i_1,\dots,i_k) = \sum_{\circ = \text{``},\text{''} \text{or} \, \text{``}+\text{''}}t^{\sigma}\zt(i_1\circ i_2 \dots \circ i_k).
\end{equation}
Thus, the series $\ztt(i_1,\dots,i_k)$ interpolates between multiple zeta values, case $t=0$, and multiple zeta star values, case $t=1$.
It turned out that the interpolated series satisfies many identities generalizing or unifying earlier result for multiple zeta and zeta star values, see for example Yamamoto~\cite{Y},
Hoffman and Ihara~\cite{HI} or Hoffman~\cite{H2000,H2018}. In particular, a so-called quasi-shuffle product~\cite{H2000,H2018,HI,Y}, sometimes also called stuffle product or harmonic shuffle product, can be defined for the interpolated zeta values. Hoffman and Ihara used an algebra framework, which leads amongst others to expressions for interpolated multiple zeta values
 $\ztt(\{m\}_{k})$ in terms of Bell polynomials and ordinary single argument zeta values, $m\ge 1$. We discuss and (re)derive in this work the results for $\ztt_{n}(\{m\}_{k})$ and $\ztt(\{m\}_{k})$ using symbolic combinatorics.  Moreover, we relate and extend certain multiple harmonic sums identities~\cite{Prode2010,ProdeTau2010} to $\ztt_{n}(\{m\}_{k})$.

\smallskip

We study in particular the interpolated truncated multiple zeta values: we are interested in the truncated series 
\begin{align}
\label{EqYamamotoOrdered}
\ztt_{n}(\{1\}_{k}) & =\sum_{\circ = \text{``},\text{''} \text{or} \, \text{``}+\text{''}}t^{\sigma}\zt_{n}(\underbrace{1 \circ 1 \circ \dots \circ 1}_{k}) %= \sum_{\mathbf{p}\in \mathcal{P}_O(k)}t^{k-\laenge(\mathbf{p})}\zt_{n}(\mathbf{p})
\end{align}
and their properties. Note that $\ztt_{n}(\{1\}_{k})$ interpolates between $\zt_n(\{1\}_k)$ and $\zts_n(\{1\}_k)$, $t=0$ and $t=1$, respectively. 

\smallskip

On the other hand, classical combinatorial problems are the enumeration of $k$-element multisets, short \myt{$k$-multisets}, of $[n]=\{1,\dots,n\}$, leading to the multiset coefficients, $\binom{n+k-1}{k}$ and the enumeration of $k$-element sets, short \myt{$k$-subsets}, of $[n]=\{1,\dots,n\}$ leading to the binomial coefficients, $\binom{n}{k}$. 

\smallskip 

The goal of this work is to unify these two topics, namely interpolated (truncated) multiple zeta values and the enumeration of $k$-subsets and $k$-multisets. We study $k$-\myt{multisets} of $[n]$ and introduce a \myt{weighted enumeration} of $k$-multisets of $[n]$ in terms of a given weight sequence $\boldsymbol{a}=(a_j)_{j\in\N}$. Our weighted enumeration also takes into account an additional parameter $\sigma=\sigma_{n,k}$. It is defined similarly to Yamamoto~\cite{Y}: $\sigma$ counts the number of equalities in the nested sum representation, corresponding to the number of plus signs in~\eqref{EqYamamoto} or~\eqref{EqYamamotoOrdered}. The weighted enumeration also has an \myt{algebraic interpretation}, connecting elementary symmetric functions $e_k$ and complete symmetric functions $h_k$; moreover power sum $p_ks$ also appear in an alternative representation.

\smallskip

In the next section we introduce our weighted enumeration of $k$-multisets, given an arbitrary sequence $(a_j)_{j\in \N}$ of positive reals, $a_j>0$. Closely related ideas have been considered by Vignat and Takhare~\cite{VW}, who considered (non-truncated) generalized multiple zeta values and infinity products, with $(a_j)_{j\in \N}$ given by the zeros of certain special functions. 
Moreover, Bachmann very briefly discussed related ideas at the end of his article~\cite{B}. Also related are the partition zeta functions
introduced by Schneider~\cite{Schneider}.

\smallskip

We study two aspects of the weighted enumerations: first, its relation to $k$-sets and $k$-multisets and to the interpolated (truncated) multiple zeta function,
leading to our main result in Theorems~\ref{PropGF}, and second, the probabilistic aspects of $\sigma$. 
While the algebraic and number theoretic aspects of interpolated multiple zeta values have been studied in the literature,
no studies so far have been conducted on its stochastic properties concerning structures
of large size and its distribution, to the best of the author's knowledge. Questions such as "what is the
expected value of the parameter $\sigma$" or "what is the limit law of $\sigma$" have not been considered before.

Introducing suitably defined random variables $S_{n,k}$, we study the distribution of $\sigma$, as well as limit laws for $k$ or $n$ tending to infinity.
In particular, we show that the parameter $\sigma$ in $k$-multisets follows a hypergeometric distribution and derive Poisson and Gaussian limit laws.
For $\zts_n(\{1\}_k)$ we show that also different limit laws occur. In a later section we introduce refinements of the parameter $\sigma$ an also discuss related random variables. 
It turns out that the random variables $S_{n,k}$, its refinements and its limit laws are closely related to many quantities studied earlier in combinatorial probability theory:
maxima in hypercubes~\cite{Bai,Hwang2004}, unsuccessfull search in binary search tree~\cite{DobSmy1996}, descendants in increasing trees~\cite{KP2006}, edge-weighted increasing trees~\cite{KP2008,KW}, distances in increasing trees~\cite{Dob1996,DobSmy1996}, leaf-isolation procedures in random trees~\cite{KP2008b}, as well as asymptotics of the Poisson distribution~\cite{DobSmy1996,Hwang1999}.

\smallskip 

\smallskip

Throughout this work use the notation $H_n=\zt_n(1)$ for the $n$th harmonic number, $H_n^{(j)}=\zt_n(j)$ for the $n$th generalized harmonic number of order $j$. We denote with $\mathcal{N}(0,1)$ a standard normal distributed random variable, 
with $\Po(\lambda)$ a Poisson distributed random variable with parameter $\lambda$, and with $\Hy(N,K,n)$ a hypergeometric distributed random variable with parameters $N,K,n$. For the reader's convenience a few basic facts about the probability distributions, appearing later in our analysis, are collected in the appendix of this work.
We denote with $\fallfak{x}{k}$ the $k$th falling factorial, $\fallfak{x}{k}=x(x-1)\dots(x-(k-1))$, $k\in\N$, with $\fallfak{x}0=1$. 
Furthermore, we use the abbreviation $E_x$ for the evaluation operator at $x=1$. Moreover, we denote by $X \law Y$ the equality in distribution of the random variables $X$ and $Y$, and by $X_{n} \claw X$ convergence in distribution of the sequence of random variables $X_{n}$ to a random variable $X$. We denote with $\mathcal{P}(k)$ the set of ordered partitions of the integer $k$, 
often also called compositions of the integer $k$, and $\laenge(\mathbf{p})$ the length of a partition $\mathbf{p}$, defined as the number of its summands. E.g., for $k=3$ we have
$\mathcal{P}_O(3)=\{(1,1,1), (2,1), (1,2),(3)\}$ and 
\[
\mathcal{P}_O(4)=\{(1,1,1,1), (1,1,2), (1,2,1),(2,1,1),(1,3),(3,1),(4)\}.
\]
Concerning length of partitions, we have $\laenge((1,1,1,1))=4$ and $\laenge((4))=1$.

\section{Interpolated weighted Multisets and parameter sigma}
\subsection{Parameter sigma}
We consider $k$-multisets of $[n]= \{1,\dots ,n\}$. Strictly speaking, we do not directly consider multisets. Instead, we induce an increasing order on the distinct elements in the multisets and consider the ordered sequences. For example, the unordered multiset $\{1,1,4,2,2,6,2\}$ corresponds to the increasing sequence $(1,1,2,2,2,4,6)$. 
Let $\M_{n,k}$ denote the set 
\begin{equation}
\M_{n,k}=\{\vec{\ell}\in\N^k\colon 1\le \ell_k\le \dots \le \ell_2\le \ell_1 \le n\},
\label{eq:Mnk}
\end{equation}
such that each $\vec{\ell}$ corresponds to a $k$-multiset of $[n]$. In a slight abuse of notation, we denote with $| \vec{\ell} |=k$ the length of $\vec{\ell}$, in other words the cardinality of the corresponding multiset. The cardinality of $\M_{n,k}$ is given by the multiset coefficients, counting the number of $k$-multisets of $[n]$: 
\[
|\M_{n,k}|=\sum_{\ell_1=1}^n \sum_{\ell_2=1}^{\ell_1}\dots \sum_{\ell_k=1}^{\ell_{k-1}}1=\binom{n+k-1}{k}.
\]
Closely related is the set 
\begin{equation}
\S_{n,k}=\{\vec{\ell}\in\N^k\colon 1\le \ell_k< \dots < \ell_2< \ell_1 \le n\},
\label{eq:Rnk}
\end{equation}
whose cardinality is given by binomial coefficients, counting the number of $k$-sets of $[n]$: 
\[
|\S_{n,k}|=\sum_{\ell_1=1}^n \sum_{\ell_2=1}^{\ell_1-1}\dots \sum_{\ell_k=1}^{\ell_{k-1}-1}1=\binom{n}{k}.
\]

\smallskip

In the following we introduce the parameter $\sigma=\sigma_{n,k}$ for elements $\vec{\ell}\in\M_{n,k}$.
\begin{defi}
Let $\vec{\ell}\in \M_{n,k}$~\eqref{eq:Mnk}. Then $\sigma=\sigma_{n,k}(\vec{\ell})$ is defined as the number
of elements of $\vec{\ell}$ equal to the preceding one in $\vec{\ell}$, or in other words
\[
\sigma(\vec{\ell}) =\sigma(\ell_1,\dots,\ell_k) = |\{1 \le  j \le k - 1\colon\,\, \ell_j = \ell_{j+1} \}|.
\]
\end{defi}

\begin{example}
Given $\vec{\ell}=(1,1,2,2,2,4,6)\in\M_{n,7}$ then $\sigma(\vec{\ell})=3$. Given $\vec{\ell}=(5,5,5,5)\in\M_{n,3}$ then $\sigma(\vec{\ell})=3$. 
\end{example}
\begin{example}
The parameter $\sigma$ relates $\M_{n,k}$ and $\S_{n,k}$, as defined in~\eqref{eq:Mnk},~\eqref{eq:Rnk}, respectively:
\[
\S_{n,k}=\{\vec{\ell}\in\M_{n,k}\colon \sigma(\vec{\ell})=0\}.
\]
\end{example}
Next we define multiplicative weights for $\vec{\ell}\in\M_{n,k}$ and the weighted enumerations of $\M_{n,k}$~\eqref{eq:Mnk}. 

\begin{defi}[Interpolated weighted multiset sums]
Given an the weight sequence $\boldsymbol{a}=(a_{j})_{j\in \N}$ of positive real numbers $a_j$. Then, 
the weight $w(\vec{\ell})=w(\vec{\ell},\boldsymbol{a})$ of $\vec{\ell}=(\ell_1,\dots,\ell_k)\in\M_{n,k}$ is defined in terms of $\boldsymbol{a}$ as $w(\vec{\ell})=\prod_{j=1}^{k}a_{\ell_{j}}$ and $\theta_{n;k}(t)=\theta_{n;k}(t;\boldsymbol{a})$ is defined as 
\begin{equation}
\label{def:theta}
\theta_{n;k}(t)=\sum_{\vec{\ell}\in\M_{n,k}}w(\vec{\ell})t^{\sigma(\vec{\ell})}
=\sum_{n \ge \ell_1\ge \cdots\ge \ell_k\ge 1}a_{\ell_1}a_{\ell_2}\dots a_{\ell_k}t^{\sigma(\vec{\ell})}.
\end{equation}
\end{defi}
Note that we use the convention $\theta_{n;0}(t)=1$ in the boundary case $k=0$.

\begin{remark}
The construction is also well defined for 
\begin{equation}
\M_{\infty,k}=\{\vec{\ell}\in\N^k\colon 1\le \ell_k\le \dots \le \ell_2\le \ell_1 <\infty\},
\label{eq:Minfty}
\end{equation}
as long as the weight sequence $\boldsymbol{a}=(a_{j})_{j\in \N}$ is chosen in such a way that
$\theta_{\infty;k}(t)=\sum_{\vec{\ell}\in\M_{\infty,k}}w(\vec{\ell})t^{\sigma(\vec{\ell})}$
converges. This construction is very closely related to the generalized multiple zeta values of Vignat and Wakhare~\cite{VW}.
\end{remark}

\begin{remark}[Interpolated complete symmetric polynomials]
\label{RemarkSym}
Regarding the weight sequence $\boldsymbol{a}=(a_{j})_{j\in \N}$ as variables, the values $\theta_{n;k}(t)$ can also be interpreted in an algebraic way. 
This is no surprise, due to the well known relationship between the multiple zeta values and quasi-symmetric functions~\cite{H2000,H2018,HI,VW,Z}.
They relate complete elementary symmetric polynomials $h_k$ and elementary symmetric polynomials $e_k$: $\theta_{n;k}(t)=\theta_{n;k}(t;\boldsymbol{a})$ at $t=0$ or $t=1$ is given by
\[
\theta_{n;k}(0;\boldsymbol{a})=e_k(a_1,\dots,a_n)=\sum_{n \ge \ell_1> \cdots> \ell_k\ge 1}a_{\ell_1}a_{\ell_2}\dots a_{\ell_k},
\]
and
\[
\theta_{n;k}(1;\boldsymbol{a})=h_k(a_1,\dots,a_n)=\sum_{n \ge \ell_1\ge \cdots\ge \ell_k\ge 1}a_{\ell_1}a_{\ell_2}\dots a_{\ell_k}.
\]
We will see later that the power sums $p_k$ on $n$ variables,
\begin{equation}
\label{eqn:powerSum}
p_k(a_1,\dots,a_n)=A_n(k)=\sum_{m=1}^{n}a_m^k,
\end{equation}
also appear in an alternative representation of $\theta_{n;k}(t)$.
\end{remark}

\begin{example}[Interpolated $k$-sets and multisets]
\label{ExampleMultiset}
Given the weight sequence $\boldsymbol{a}=(1)_{j\in\N}$ we obtain 
\[
\theta_{n;k}(t)=\sum_{n \ge \ell_1\ge \cdots\ge \ell_k\ge 1}t^{\sigma(\vec{\ell})},
\]
interpolating between  $k$-sets and multisets of $[n]$.
The special values $\theta_{n;k}(0)$ and $\theta_{n;k}(1)$ enumerate $k$-sets~\eqref{eq:Rnk} and multisets~\eqref{eq:Mnk} of $[n]$:
\[
\theta_{n;k}(0)=|\S_{n,k}|=\binom{n}{k},\quad \theta_{n;k}(1)=|\M_{n,k}|=\binom{n+k-1}{k}.
\]
\end{example}

\smallskip

\begin{example}
Given the weight sequence $\boldsymbol{a}=(j)_{j\in\N}$ we obtain 
\[
\theta_{n;k}(t)=\sum_{n \ge \ell_1\ge \cdots\ge \ell_k\ge 1}\ell_1\cdots \ell_k\cdot t^{\sigma(\vec{\ell})}.
\]
The special values are given in terms of the Stirling numbers of the first and second kind~\cite{GraKnuPa}:
\[
\theta_{n;k}(0)=\stir{n+1}{n+1-k}=n!\cdot\zt_{n}(\{1\}_{n+k}),\quad \theta_{n;k}(1)=\Stir{n+k}{n}.
\]
%
%
%$\Theta_{n}(z,t)=\prod_{m=1}^n\Big(1+\frac{m z}{1-m z t}\Big)$. 
%\begin{align*}
%\theta_{n;k}(0)&= [z^k]\Theta_{n}(z,0)=[z^k]\prod_{m=1}^n(1+mz)
%=[z^k]z^n\prod_{m=1}^n(1/z+m)\\
%&=[z^k]z^{n+1}\prod_{m=0}^n(1/z+m)
%=[z^k]z^{n+1}\sum_{j=0}^{n+1}\stir{n+1}{j}\frac{1}{z^j}\\
%&=[z^k]\sum_{j=0}^{n+1}\stir{n+1}{j}z^{n+1-j}=\stir{n+1}{n+1-k}\\
%&=n!\zt_{n}(\{1\}_{n+k}).
%\end{align*}
%
%
%\begin{align*}
%\theta_{n;k}(1)&= [z^k]\Theta_{n}(z,1)=[z^k]\prod_{m=1}^n\frac{1}{1-m z }
%=[z^k]\sum_{j\ge n} \Stir{j}{n}z^{j-n}=\Stir{n+k}{n}.
%\end{align*}

\end{example}

\begin{example}[Interpolated truncated multiple zeta values]
Given the weight sequence $\boldsymbol{a}=(\frac1{j^{m}})_{j\in\N}$, $m>0$ we obtain 
a special instance of the truncated interpolated multiple zeta values,
\[
\theta_{n;k}(t)=\ztt_n(\{m\}_k)=\sum_{n \ge \ell_1\ge \cdots\ge \ell_k\ge 1}\frac{t^{\sigma(\vec{\ell})}}{\ell_1^{m}\cdots \ell_k^{m}},
\]
such that $\theta_{n;k}(0)=\zt_n(\{m\}_k)$ and $\theta_{n;k}(1)=\zts_n(\{m\}_k)$.
\end{example}

\subsection{Ordered partitions and Interpolated multiple zeta values}
Let $\vec{\ell}\in\M_{n,k}$~\eqref{eq:Mnk}. We can associate an ordered partition $\mathbf{p}\in \mathcal{P}_O(k)$ to $\vec{\ell}$ as follows: there exist integers $1\le s\le k$ and $n\ge j_1>j_2>\dots>j_s\ge 1$ with multiplicities
$r_1,\dots,r_s\in \N$ with $\sum_{i=1}^{s}r_i=k$ such that $\vec{\ell}=\vec{j}=(j_s^{r_s},\dots,j_1^{r_1})$. Here, we use that shorthand notation $j_i^{r_i}$ for $j_i$ appearing exactly $r_i$ times. We refer to the ordered partition $\mathbf{p}=\mathbf{p}(\vec{j})=(r_1,\dots,r_s)\in \mathcal{P}_O(k)$ as the shape of $\vec{j}$. The shape, as a map, is a surjection from $\M_{n,k}$ to $\mathcal{P}_O(k)$. 
Note that 
\[
\sigma(\vec{j})=\sum_{i=1}^{s}(r_i-1)=k-s=k-\laenge(\mathbf{p}).
\]  
We may write 
\[
\theta_{n;k}(t)=\sum_{\vec{\ell}\in\M_{n,k}}a_{\ell_1}\cdot\dots\cdot a_{\ell_k} t^{\sigma(\vec{\ell})}
=\sum_{\mathbf{p}\in \mathcal{P}_O(k)}\sum_{\substack{\vec{\ell}\in\M_{n,k}\\\text{shape}(\vec{\ell})=\mathbf{p} } } a_{\ell_1}\cdot\dots\cdot a_{\ell_k} t^{\sigma(\vec{\ell})}
\]
Thus, for $\boldsymbol{a}=(\frac1{j^{m}})_{j\in\N}$ we get
\begin{equation*}
\begin{split}
\theta_{n;k}(t)
&= \sum_{\mathbf{p}\in \mathcal{P}_O(k)}\sum_{\substack{\vec{j}\in\M_{n,k}\\\text{shape}(\vec{j})=(r_1,\dots,r_s)=\mathbf{p} } } \frac1{j_1^{m\cdot r_1}}\cdot\dots\cdot \frac1{j_s^{ m\cdot r_s }} t^{\sigma(\vec{j})}\\
&= \sum_{\mathbf{p}\in \mathcal{P}_O(k)}\sum_{\substack{n\ge j_1>j_2>\dots>j_s\ge 1\\\text{shape}(\vec{j})=(r_1,\dots,r_s)=\mathbf{p} } } \frac1{j_1^{m\cdot r_1}}\cdot\dots\cdot \frac1{j_s^{ m\cdot r_s }} t^{\sigma(\vec{j})}\\
&=\sum_{\mathbf{p}\in \mathcal{P}_O(k)}\zt_n(m\cdot r_1,\dots,m\cdot r_s)t^{k-s}=\sum_{\mathbf{p}\in \mathcal{P}_O(k)}\zt_n(m\cdot \mathbf{p})t^{k-\laenge(p)}.
\end{split}
\end{equation*}

\section{Generating functions and Bell polynomials}
In this section we use symbolic combinatorics to obtain the generating function $\Theta_{n}(z,t)=\sum_{k\ge 0}\theta_{n;k}(t)z^k$.
This leads then to several expressions for $\theta_{n;k}(t)$, collected in our main result, Theorem~\ref{PropGF},
generalizing several results in the literature~\cite{Bati,D,HI,Prode2010,KP}. To the best of the author's knowledge, our general results have not appeared before
in literature, although related ideas have been developed earlier or in parallel, as mentioned in the introduction.
Let $\M_n$ denote all multisets of $[n]$, $\M_n=\bigcup_{k=1}^{\infty}\M_{n,k}$, with $\M_{n,k}$ as stated in~\eqref{eq:Mnk}. In order to gain more insight into $\theta_{n;k}(t)$ we use symbolic combinatorics and study the generating function 
\begin{equation*}
\begin{split}
\Theta_{n}(z,t)&=\sum_{\vec{\ell}\in\M_{n}}w(\vec{\ell})t^{\sigma(\vec{\ell})}z^{|\vec{\ell}|}
=\sum_{k\ge 0}\theta_{n;k}(t)z^k\\
&=\sum_{k\ge 0}z^k\sum_{n \ge \ell_1\ge \cdots\ge \ell_k\ge 1}a_{\ell_1}a_{\ell_2}\dots a_{\ell_k}t^{\sigma(\vec{\ell})}.
\end{split}
\end{equation*}

\begin{thm}[Generating functions and power sums]
\label{ThmGF}
The generating function $\Theta_{n}(z,t)$ is given by
\begin{equation}
\label{eq:ThmGF1}
\Theta_{n}(z,t)=\prod_{m=1}^n\Big(1+\frac{a_m z}{1-a_m z t}\Big),
\end{equation}
as well as 
\begin{equation}
\label{eq:ThmGF2}
\Theta_{n}(z,t)=\exp\Big(\sum_{j=1}^{\infty}\frac{z^j}{j}A_n(j)\big( t^j-(t-1)^j\big)\Big),
\end{equation}
with $A_n(j)$ as in~\eqref{eqn:powerSum}. Moreover, assume that $t\neq 0$ and that the $a_1,\dots,a_n>0$ are all distinct. 
Then, 
\begin{equation}
\label{eq:ThmGF3}
\Theta_{n}(z,t)=\frac{(t-1)^n}{t^n}+(-1)^n\sum_{j=1}^n\bigg(\frac{\prod_{m=1}^{n}\big(\frac{1-t}{a_j t}+\frac{1}{a_m}\big) }{
\prod_{\substack{\ell=1\\ \ell\neq j}}^n (\frac{1}{a_j }-\frac{1}{a_\ell })} \bigg)\cdot\frac{1}{zt-\frac{1}{a_k }}. 
\end{equation}
\end{thm}

\begin{remark}
Very similar results also hold for $n\to\infty$ and infinite multisets $\M_{\infty,k}$~\eqref{eq:Minfty}.
In special cases it is possible to derive the complete generating function
$T(x,z,t)=\sum_{n\ge 0}\Theta_n(z,t)x^n=\sum_{n\ge 0}\sum_{k\ge 0}\theta_{n;k}(t) x^n z^k$. 
For $(a_j)=(1)_{j\in\N}$ we get
\[
T(x,z,t)=\frac{1-zt}{(1-zt)(1-x)-zx}.
\]
For $(a_j)=(\frac1j)$ we get an ordinary hypergeometric function
\[
T(x,z,t)={}_2F_1(1,1-z(t-1),1-z t;x).
\]
Similar, but more involved, partial fraction decompositions exist when the $a_m$ are not all distinct.
\end{remark}

Before we state the proof of Theorem~\ref{ThmGF} we collect a result on partial fraction decomposition.
\begin{lem}[Partial fraction decomposition]
\label{Lemma:parFrac}
Let $p_n(z)$ and $q_n(z)$ denote monic polynomials of degree $n$ with distinct zeros given by $-\alpha_k$ and $\beta_k$, $1\le k\le n$, respectively, 
such that $\{-\alpha_1,\dots,-\alpha_n\}\cap \{\beta_1,\dots,\beta_n\}=\emptyset$.
The rational function $p_n(z)/q_n(z)$ has the partial fraction decomposition
\[
\frac{p_n(z)}{q_n(z)}
=\prod_{k=1}^{n}\frac{z+\alpha_k}{z-\beta_k}
=1+\sum_{j=1}^n\bigg(\frac{p_n(\beta_j)}{\prod_{\substack{\ell=1\\ \ell\neq j}}^n (\beta_j-\beta_\ell)}\bigg)\cdot\frac{1}{z-\beta_j}.
\]
\end{lem}

\begin{proof}[Proof of Theorem~\ref{ThmGF}]
First, we derive~\eqref{eq:ThmGF1} using symbolic constructions from analytic combinatorics, see Flajolet and Sedgewick~\cite{FS2009}.
Let $\mathcal{Z}_m=\{m\}$ be a combinatorial class of size one, $1\le m\le n$. Due to the sequence construction we can describe the class of multisets $\mathcal{B}_m$ of $\mathcal{Z}_m$ as follows
\[
\mathcal{B}_m=\text{SEQ}(\mathcal{Z}_m)=\{\epsilon\} + \mathcal{Z}_m+ \mathcal{Z}_m\times \mathcal{Z}_m+\mathcal{Z}_m\times \mathcal{Z}_m\times \mathcal{Z}_m+\dots;
\]
Thus, the generating function 
\[
B_m(z)=\sum_{\beta\in\mathcal{B}_m}w(\beta)t^{\sigma(\beta)}z^{|\beta|}= 1+\sum_{\epsilon\neq \beta\in\mathcal{B}_m}w(\beta)t^{|\beta|-1}z^{|\beta|}
\] 
is given by
\[
B_m(z)=1+\sum_{j=1}^{\infty}t^{j-1}a_m^j z^{j}=1+\frac{a_mz}{1-a_m t z}.
\]
All multisets $\M_n=\bigcup_{k=1}^{\infty}\M_{n,k}$ of $[n]$ can be combinatorially described by
\[
\M_n=\mathcal{B}_1\times \mathcal{B}_2\times\dots \times \mathcal{B}_n.
\]
Hence, the generating function $\Theta_n(z,t)$ is given by
\[
\Theta_n(z,t)=\prod_{m=1}^{n}B_m(z)=\prod_{m=1}^{n}\left(1+\frac{a_mz}{1-a_m t z}\right).
\]
Next, we use the $\exp-\log$ representation to obtain the expression~\eqref{eq:ThmGF2} for $\Theta_{n}(z,t)$.
\begin{equation*}
\begin{split}
\Theta_{n}(z,t)&=\prod_{m=1}^n\Big(1+\frac{a_m z}{1-a_m z t}\Big)= \prod_{m=1}^n\exp\Big(\log\big(1+\frac{a_m z}{1-a_m z t}\big)\Big)\\
%&=\exp\Big(\sum_{m=1}^n\log\big(1+\frac{a_m z}{1-a_m z t}\big)\Big)=\exp\Big(\sum_{m=1}^n\log\big(\frac{1-a_m z (t-1)}{1-a_m z %t}\big)\Big)\\
&=\exp\Big(\sum_{m=1}^n\log\big(1-a_m z (t-1)\big) - \sum_{m=1}^n\log\big(1-a_m z t\big)\Big).
\end{split}
\end{equation*}
We use~\eqref{eqn:powerSum} and expand the two logarithm functions:
\begin{equation*}
\begin{split}
\Theta_{n}(z,t)&=\exp\Big(\sum_{m=1}^n\sum_{j=1}^{\infty}\frac{a_m^j z^j}{j}\big( t^j-(t-1)^j\big)\Big)=\exp\Big(\sum_{j=1}^{\infty}\frac{z^j}{j}A_n(j)\big( t^j-(t-1)^j\big)\Big).
\end{split}
\end{equation*}
Finally, we turn to~\eqref{eq:ThmGF3}. We assume that $t\neq 0$ and that the $a_1,\dots,a_n>0$ are all distinct. Then, we can write $\Theta_n(z,t)$ 
as follows:
\[
\Theta_n(z,t)=\frac{(t-1)^n}{t^n}\cdot\prod_{m=1}^{n}\frac{z+\frac{1}{a_m(1-t)}}{z-\frac{1}{a_m t}}.
\]
Consequently, applying partial fraction decomposition (see Lemma~\ref{Lemma:parFrac}) with $\alpha_k=\frac{1}{a_k(1-t)}$ and $\beta_k=\frac{1}{a_k t}$ 
gives
\[
\Theta_n(z,t)=\frac{(t-1)^n}{t^n}\bigg(1+\sum_{j=1}^n\bigg(\frac{\prod_{m=1}^{n}\big(\frac{1}{a_j t}+\frac{1}{a_m(1-t)}\big) }{
\prod_{\substack{\ell=1\\ \ell\neq j}}^n (\frac{1}{a_j t}-\frac{1}{a_\ell t})} \bigg)\cdot\frac{1}{z-\frac{1}{a_j t}}\bigg).
\]
Multiplying the sum with $\frac{(t-1)^n}{t^n}$ directly gives the stated expression~\eqref{eq:ThmGF3}.
\end{proof}

As a consequence of the theorem before, we obtain alternative expressions for the values $\theta_{n;k}(t)$~\eqref{def:theta}.
In the following we use the complete Bell polynomials $B_n(x_1,\dots,x_n)$, which are defined via the identity
\begin{equation}
\label{def:BellPoly}
\exp\Big(\sum_{\ell\ge 1}\frac{z^{\ell}}{\ell!}x_\ell\Big)
= \sum_{j\ge 0}\frac{B_j(x_1,\dots,x_j)}{j!}z^j.
\end{equation}

\begin{thm}
\label{PropGF}
The values $\theta_{n;k}(t)$~\eqref{def:theta} can be expressed in various ways:
\begin{itemize}
	\item in terms of values $\theta_{n;k}(0)$ and $\theta_{n;k}(1)$:
\begin{equation}
\theta_{n;k}(t)=\sum_{j=0}^{k}t^j\theta_{n,j}(1)\cdot (1-t)^{k-j}\theta_{n,k-j}(0).
%\sum_{j=0}^{k-1}t^j\cdot \sum_{r=0}^{j}(-1)^{j-r}\binom{k-r}{j-r}\cdot\theta_{n;r}(1)\cdot\theta_{n,k-r}(0).
\label{eq:PropGF1}
\end{equation}
\item Bell polynomials and the power sums $A_n(j)$~\eqref{eqn:powerSum}:
\begin{equation}
\begin{split}
\label{eq:PropGF2}
\theta_{n;k}(t)&=\frac1{k!} B_k(0!A_n(1)\big( t^1-(t-1)^1\big),\dots,(k-1)!A_n(k)\big( t^k-(t-1)^k\big))\\
&=\sum_{m_1+2m_2+\dots=k}\frac{1}{m_1!m_2!\dots}\Big(\frac{A_n(1)( t^1-(t-1)^1)}{1}\Big)^{m_1} \Big(\frac{A_n(2)( t^2-(t-1)^2)}{2}\Big)^{m_2}\dots.
\end{split}
\end{equation}
\item Determinantal expression: let $\alpha_n(j;t)=A_n(j)( t^j-(t-1)^j)$, it holds
\begin{equation}
\label{eq:PropGF3}
\theta_{n;k}(t)=\frac{1}{k!}\left|
\begin{matrix}
\alpha_n(1;t) & -1 & 0 & \dots &0\\
\alpha_n(2;t) & \alpha_n(1;t) & -2 & \dots &0\\
\hdots & \hdots & \hdots & \vdots &\hdots\\
\alpha_n(k-1;t) & \alpha_n(k-2;t)& \alpha_n(k-3;t) & \dots &-(k-1)\\
\alpha_n(k;t) & \alpha_n(k-1;t) & \alpha_n(k-2;t) & \dots &\alpha_n(1;t)\\
\end{matrix}
\right|.
\end{equation}
\item Combinatorial sum: assume that $t\neq 0$ and that the values $a_1,\dots,a_n>0$ are all distinct. 
Then, 
\begin{equation}
\label{eq:PropGF4}
\theta_{n;k}(t)%=[z^k](-1)^n\sum_{j=1}^n\bigg(\frac{\prod_{m=1}^{n}\big(\frac{1-t}{a_j t}+\frac{1}{a_m}\big) }{
%\prod_{\substack{\ell=1\\ \ell\neq j}}^n (\frac{1}{a_j }-\frac{1}{a_\ell })} \bigg)\cdot\frac{1}{zt-\frac{1}{a_j }}
%=[z^k](-1)^n\sum_{j=1}^n\bigg(\frac{\prod_{m=1}^{n}\big(\frac{1-t}{a_j t}+\frac{1}{a_m}\big) }{
%\prod_{\substack{\ell=1\\ \ell\neq j}}^n (\frac{1}{a_j }-\frac{1}{a_\ell })} \bigg)\cdot\frac{-a_j}{1-a_jzt}. 
=(-1)^{n-1}\sum_{j=1}^n\bigg(\frac{\prod_{m=1}^{n}\big(\frac{1-t}{a_j t}+\frac{1}{a_m}\big) }{
\prod_{\substack{\ell=1\\ \ell\neq j}}^n (\frac{1}{a_j }-\frac{1}{a_\ell })} \bigg)\cdot a_j^{k+1}t^k. 
\end{equation}
\end{itemize}
\end{thm}

\begin{example}
For $t=\frac12$ we solely sum over the odd indices $m_1,m_3,\dots$, since only $m_2=m_4=\dots=0$ lead to a positive contribution.
We obtain 
\begin{equation*}
\begin{split}
\theta_{n;k}\big(\frac12\big)&=\frac1{k!} B_k\Big(0!A_n(1)\cdot 2\cdot \frac1{2^1},0,2!A_n(1)\cdot 2\cdot \frac1{2^3}\dots,(k-1)!A_n(k)\cdot\big( \frac1{2^k}-(-\frac12)^k\big)\Big)\\
&=\sum_{m_1+3m_3+\dots=k}\frac{2^{m_1+m_3+\dots}}{2^{1\cdot m_1+3\cdot m_3+\dots}m_1!m_3!\dots}\Big(\frac{A_n(1)}{1}\Big)^{m_1} \Big(\frac{A_n(3)}{3}\Big)^{m_3}\dots.
\end{split}
\end{equation*}
\end{example}

A direct byproduct of this result is an expression for $\ztt_n(\{m^k\})$. We state the formula for $\ztt_n(\{1^k\})$, generalizing the already known results for $t=0$, $\zt_n(\{1\}_k)$ and $t=1$, $\zts_n(\{1\}_k)$; see for example~\cite{KP} and the references therein. 

\begin{example}
\label{CorollZeta}
Given the weight sequence $\boldsymbol{a}=(\frac1{j})_{j\in\N}$ we obtain for 
$\theta_{n;k}(t)=\ztt_n(\{1\}_k)$ the results
\[
\ztt_n(\{1\}_k)=\sum_{j=0}^{k}t^j\zts_n(\{1\}_j)\cdot (1-t)^{k-j}\zt_n(\{1\}_{k-j}).
\]
as well as
\begin{equation*}
\begin{split}
\ztt_n(\{1\}_k)&=\frac1{k!} B_k(0!H_n^{(1)}\big( t^1-(t-1)^1\big),\dots,(k-1)!H_n^{(k)}\big( t^k-(t-1)^k\big))\\
&=\sum_{m_1+2m_2+\dots=k}\frac{1}{m_1!m_2!\dots}\Big(\frac{H_n^{(1)}( t^1-(t-1)^1)}{1}\Big)^{m_1} \Big(\frac{H_n^{(2)}( t^2-(t-1)^2)}{2}\Big)^{m_2}\dots.
\end{split}
\end{equation*}
Moreover, we obtain for $t\neq 0$ the identity
\[
\ztt_n(\{1\}_k)=t^k\cdot\sum_{j=1}^{n}\binom{n+\frac{1-t}{t}\cdot j}{n}\binom{n}{j}\frac{(-1)^{j-1}}{j^k}.
\]
For $t=\frac12$ we obtain as a special case a truncated analog of the formula of Hoffman and Ihara~\cite[eqn.~$(41)$]{HI}:
\begin{align*}
\zt^{\frac12}_n(\{1\}_k)
&=\frac{1}{2^k k!}(B_{k}(0!\cdot 2\Ha{n}{1},0,2!\cdot 2\Ha{n}{3},0,\dots)\\
&=\sum_{m_1+3m_3+5m_5\dots=k}\frac{2^{m_1+m_3+m_5+\dots}}{2^{k} m_1!m_3!\dots}\Big(\frac{\Ha{n}{1}}{1}\Big)^{m_1} \Big(\frac{\Ha{n}{3}}{3}\Big)^{m_3}\dots.
\end{align*}
The combinatorial sum extends a result of Prodinger for $t=\frac12$~\cite{Prode2010} (see also~\cite{Prode2008,ProdeTau2010,Verma}), 
originally stated solely in terms of the combinatorial sum and the expression in terms of Bell polynomials~\eqref{def:BellPoly}:
\[
\zt^{\frac12}_n(\{1\}_k)=\frac{1}{2^k}\sum_{j=1}^{n}\binom{n+j}{j}\binom{n}{j}(-1)^{j-1}\frac{1}{j^k}.
\]
%\[
%\theta_{n;k}(t)=\frac1{2^k}(-1)^{n-1}\sum_{j=1}^n\bigg(\frac{\prod_{m=1}^{n}(j+m) }{
%\prod_{\substack{\ell=1\\ \ell\neq j}}^n (j-\ell) \bigg)\cdot \frac1{j^{k+1}}. 
%\]
%(n+j)!/j! * 1/ (j-1)! *1/(n-j)! *(-1)^(n-j) /j  
Finally, we mention that for $t=1$ the combinatorial sum turns into a well-known and often rediscovered formula~\cite{Bati,D,FS,KP}:
\[
\zts_n(\{1\}_k)=\sum_{j=1}^{n}\binom{n}{j}(-1)^{j-1}\frac{1}{j^k}.
\]
\end{example}

\smallskip

\begin{proof}[Proof of Theorem~\ref{PropGF}]
First, we turn to the expression~\eqref{eq:PropGF1} for $\theta_{n,k}(t)$.
From Theorem~\ref{ThmGF} and~\eqref{eq:ThmGF1} we obtain 
\begin{align*}
\Theta_{n}(z,t)&
=\prod_{m=1}^n\Big(1+\frac{a_m z}{1-a_m z t}\Big)
=\prod_{m=1}^n\frac{1+a_m z(1-t)}{1-a_m z t}\\
&=\bigg(\sum_{j\ge 0 }(zt)^j\theta_{n,j}(1)\bigg)\cdot
\bigg(\sum_{j\ge 0 }z^j(1-t)^j\theta_{n,j}(0)\bigg).
\end{align*}
Thus, extraction of coefficients gives the stated result:
\begin{align*}
\theta_{n,k}(t)&
=[z^k]\bigg(\sum_{j\ge 0 }(z t)^j\theta_{n,j}(1)\bigg)\cdot
\bigg(\sum_{j\ge 0 }z^j(1-t)^j\theta_{n,j}(0)\bigg)\\
&=\sum_{j=0}^{k}t^j\theta_{n,j}(1)\cdot (1-t)^{k-j}\theta_{n,k-j}(0).
\end{align*}

\smallskip

For~\eqref{eq:PropGF2} we use~\eqref{eq:ThmGF2} of Theorem~\ref{ThmGF}. The definition of the Bell polynomials~\eqref{def:BellPoly} 
and extraction of coefficients, $\theta_{n,k}(t)=[z^k]\Theta_{n}(z,t)$, then directly gives the stated results,
with the Bell polynomials~\eqref{def:BellPoly} evaluated at $x_\ell=(\ell-1)!A_n(\ell)\big( t^\ell-(t-1)^\ell\big)$.
Following Hoffman~\cite{H2017}, the determinant~\eqref{eq:PropGF3} can be obtained using the theory of symmetric functions. 
Let $x_1,x_2,\dots$ denote variables of degree one. As in Remark~\ref{RemarkSym} let $h_k$ denote the complete symmetric functions and $p_k$ the power sums. There exists polynomials $Q_k$ such that
\[
h_k=Q_k(p_1,\dots,p_k).
\]
Let $H(z)$ and $P(z)$ denote the generating functions 
of $h_k$ and $p_k$:
\[
H(z)=\sum_{k=0}^{\infty}h_k z^k=\prod_{i\ge 1}\frac{1}{1-z x_i},
\quad 
P(z)=\sum_{k=1}^{\infty}p_k z^k=\sum_{i\ge 1}\frac{x_i}{1-z x_i}.
\]
Then, 
\[
H(z)=\exp\left(\int_{0}^{z}P(v)dv \right)=\exp\left(-\sum_{i\ge 1}\ln(1-z x_i)\right)
%=\exp\left(\sum_{j\ge 1}z^j  \cdot \frac1j\sum_{i\ge 1}x_i^j\right).
=\exp\left(\sum_{j\ge 1}z^j \cdot \frac{p_j}j\right).
\]
This gives the Bell polynomial type expression for the $Q_k$:
\[
Q_k(p_1,\dots,p_k)=\sum_{m_1+2m_2+\dots =k}\frac{1}{m_1!m_2!\dots}
\Big(\frac{p_1}{1}\Big)^{m_1}\Big(\frac{p_2}{2}\Big)^{m_2}\dots .
\]
On the other hand, MacDonald~\cite{MacDo} gives the determinant
\[
Q_k(y_1,\dots,y_k)=\frac{1}{k!}\left|
\begin{matrix}
y_1 & -1 & 0 & \dots &0\\
y_2 & y_1 & -2 & \dots &0\\
\hdots & \hdots & \hdots & \vdots &\hdots\\
y_{k-1} & y_{k-2} & y_{k-3} & \dots &-(k-1)\\
y_{k} & y_{k-1} & y_{k-2} & \dots &y_1\\
\end{matrix}
\right|,
\]
which proves the stated result. 
Finally, the combinatorial expression~\eqref{eq:PropGF4} is readily obtained by extraction of coefficients from the partial fraction decomposition~\eqref{eq:ThmGF3} of $\Theta_n(z,t)$
in Theorem~\ref{ThmGF}.
\end{proof}

\section{Further generalizations}
\subsection{Refinements of the parameter \texorpdfstring{$\sigma$}{sigma}}
\begin{defi}[Refined parameter sigma]
Let $\vec{\ell}\in \M_{n,k}$~\eqref{eq:Mnk}. Then $\sigma^{(i)}(\vec{\ell})=\sigma^{(i)}_{n,k}(\vec{\ell})$ is defined as the number equal signs in $\vec{\ell}$ stemming from numbers $i\in\N$. Then, as functions acting from $\M_{n,k}$ to $\N_0$:
\[
\sigma=\sum_{i=1}^{n}\sigma^{(i)}.
\]
\end{defi}
\begin{example}
Given $\vec{\ell}=(1,1,2,2,2,4,6)\in\M_{n,7}$ then $\sigma^{(1)}=1$, $\sigma^{(2)}=2$ and $\sigma^{(i)}=0$, $i>2$.  
\end{example}

In the following we use the vector notation 
$\vec{t}^{\vec{\sigma}(\vec{\ell})}=t_1^{\sigma_1(\vec{\ell})}\dots t_n^{\sigma_n(\vec{\ell})}$.

\begin{defi}[Multi-interpolated weighted multiset sums]
Given the weight sequence $\boldsymbol{a}=(a_{j})_{j\in \N}$. 
Let $\theta_{n;k}(\vec{t})=\theta_{n;k}(\vec{t};\boldsymbol{a})$ be defined 
as 
\[
\theta_{n;k}(\vec{t})=\sum_{\vec{\ell}\in\M_{n,k}}w(\vec{\ell})\vec{t}^{\vec{\sigma}(\vec{\ell})}
:=\sum_{n \ge \ell_1\ge \cdots\ge \ell_k\ge 1}a_{\ell_1}a_{\ell_2}\dots a_{\ell_k}t_1^{\sigma^{(1)}(\vec{\ell})}\dots t_n^{\sigma^{(n)}(\vec{\ell})}.
\]
\end{defi}

The generating function $\Theta_{n}(z,\vec{t})=\sum_{k\ge 0}\theta_{n;k}(\vec{t})z^k=\sum_{\ell\in\M_n}w(\vec{\ell})\vec{t}^{\vec{\sigma}(\vec{t})}z^{|\vec{\ell}|}$ can readily be obtained using the symbolic methods.
\begin{thm}[Generating functions and power sums]
The generating function $\Theta_{n}(z,\vec{t})$ is given by
\begin{equation*}
\begin{split}
\Theta_{n}(z,\vec{t})&=\prod_{m=1}^n\Big(1+\frac{a_m z}{1-a_m z t_m}\Big)
%&=\exp\Big(\sum_{m=1}^n\sum_{j=1}^{\infty}\frac{a_m^j z^j}{j}\big( t_m^j-(t_m-1)^j\big)\Big)\\
=\exp\Big(\sum_{j=1}^{\infty}\frac{z^j}{j}\sum_{m=1}^{n}a_m^j\big( t_m^j-(t_m-1)^j \big)\Big).
\end{split}
\end{equation*}
Moreover, assume that for $1\le m \le n$ we have $a_m\cdot t_m>0$ and distinct. 
Then, 
\[
\Theta_n(z,\vec{t})=\prod_{m=1}^{n}\frac{t_m-1}{t_m}+(-1)^n\sum_{j=1}^n\bigg(\frac{\prod_{m=1}^{n}\big(\frac{1-t_m}{a_j t_j}+\frac{1}{a_m}\big) }{\prod_{\substack{\ell=1\\ \ell\neq j}}^n (\frac{t_\ell}{a_j t_j}-\frac{1}{a_\ell}} \bigg)\cdot\frac{1}{z t_j-\frac{1}{a_j}}\bigg).
\]
\end{thm}

The proof of the theorem is very similar to the proof of Theorem~\ref{PropGF} and therefore omitted.
%\begin{proof}
%We use again the symbolic constructions from analytic combinatorics.
%Let $\mathcal{Z}_m=\{m\}$ be a combinatorial class of size one, $1\le m\le n$.
%Thus, the generating function 
%\[
%B_m(z,t_m)=\sum_{\beta\in\mathcal{B}_m}w(\beta)t_m^{\sigma(\beta)}z^{|\beta|}= 1+\sum_{\epsilon\neq \beta\in\mathcal{B}_m}w(\beta)t_m^{|\beta|-1}z^{|\beta|}
%\] 
%is given by
%\[
%B_m(z,t_m)=1+\sum_{j=1}^{\infty}t_m^{j-1}a_m^j z^{j}=1+\frac{a_mz}{1-a_m t_m z}.
%\]
%All multisets $\M_n=\bigcup_{k=1}^{\infty}\M_{n,k}$ of $[n]$ can be combinatorially described by
%\[
%\M_n=\mathcal{B}_1\times \mathcal{B}_2\times\dots \times \mathcal{B}_n.
%\]
%Hence, the generating function $\Theta_n(z,\vec{t})$ is given by
%\[
%\Theta_n(z,\vec{t})=\prod_{m=1}^{n}B_m(z,t_m)=\prod_{m=1}^{n}\left(1+\frac{a_mz}{1-a_m t_m z}\right).
%\]
%
%\smallskip 
%
%For the partial fraction decomposition we rewrite $\Theta_n(z,\vec{t})$ 
%as follows:
%\[
%\Theta_n(z,\vec{t})=\frac{\prod_{m=1}^{n}(t_m-1)}{\prod_{m=1}^{n}t_m}\cdot\prod_{m=1}^{n}\frac{z+\frac{1}{a_m(1-t_m)}}{z-\frac{1}{a_m t_m}}.
%\]
%Consequently, applying partial fraction decomposition 
%with $\alpha_k=\frac{1}{a_k(1-t_k)}$ and $\beta_k=\frac{1}{a_k t_k}$ 
%gives
%\[
%\Theta_n(z,\vec{t})=\frac{\prod_{m=1}^{n}(t_m-1)}{\prod_{m=1}^{n}t_m}\bigg(1+\sum_{j=1}^n\bigg(\frac{\prod_{m=1}^{n}\big(\frac{1}{a_j t_j}+\frac{1}{a_m(1-t_m)}\big) }{\prod_{\substack{\ell=1\\ \ell\neq j}}^n (\frac{1}{a_j t_j}-\frac{1}{a_\ell t_\ell})} \bigg)\cdot\frac{1}{z-\frac{1}{a_j t_j}}\bigg).
%\]
%Simplifications lead to the stated result.
%
%\end{proof}

\begin{example}[Even-odd interpolated truncated multiple zeta functions]
For even indices $m$ let $t_m=t_E$, whereas for odd indices let $t_m=t_O$. 
Then,
\[
\theta_{n;k}(t_E,t_O)=\sum_{\vec{\ell}\in\M_{n,k}}w(\vec{\ell})t^{\sigma^{E}(\vec{\ell})}t^{\sigma^{O}(\vec{\ell})}
\]
The generating function is
\[
\Theta_{n}(z,t_E,t_O)=\exp\Big(\sum_{j=1}^{\infty}\frac{z^j}{j}\Big[\big(t_O^j-(t_O-1)^j \big)\sum_{\substack{m=1\\ m\text { odd}}}^{n}a_m^j+\big(t_E^j-(t_E-1)^j \big)\sum_{\substack{m=1\\ m\text{ even}}}^{n}a_m^j\Big]\Big).
\]
We consider for example $a_m=\frac{1}m$, and assume $n=2N$.
Then,
\[
\sum_{\substack{m=1\\ m\text {even}}}^{n}a_m^j=\frac1{2^j}\sum_{m=1}^{N}\frac{1}{m^j}=\frac1{2^j}H_N^{(j)}
\]
and
\[
\sum_{\substack{m=1\\ m\text{ odd}}}^{n}a_m^j=\sum_{m=1}^{N}\frac{1}{(2m-1)^j}=H_{2N}^{(j)}-\frac1{2^j}H_N^{(j)}.
\]
Thus,
\[
\Theta_{2N}(z,t_E,t_O)=\exp\Big(\sum_{j=1}^{\infty}\frac{z^j}{j}\Big[\big(t_O^j-(t_O-1)^j \big)(H_{2N}^{(j)}-\frac1{2^j}H_N^{(j)})
+\big(t_E^j-(t_E-1)^j \big)\frac1{2^j}H_N^{(j)}\Big]\Big).
\]
\end{example}

\subsection{Sums of multisets: partitions}
Yet another direction would be to introduce a new parameter $p=p_{n,k}$ on $\M_{n,k}$~\eqref{eq:Mnk}, 
with $p$ the one-norm of $\vec{\ell}$. It measures the numbers $p(\vec{\ell})$ the partitions $\vec{\ell}\in\M_{n,k}$ induce:
$p(\vec{\ell})=\Vert \vec{\ell}\Vert_1=\sum_{i=1}^{k}\ell_k$. Then, 
\[
\theta_{n;k}(t,q)=\sum_{\vec{\ell}\in\M_{n,k}}w(\vec{\ell})q^{p(\vec{\ell})}t^{\sigma(\vec{\ell})}
=\sum_{n \ge \ell_1\ge \cdots\ge \ell_k\ge 1}a_{\ell_1}a_{\ell_2}\dots a_{\ell_k}q^{\sum_{i=1}^{k}\ell_k} t^{\sigma(\vec{\ell})}.
\]
Note that this can alternatively be achieved by a change of weights $a_m\mapsto a_m\cdot q^m$. 
Then, the generating function $\Theta_{n}(z,t,q)=\sum_{k\ge 0}\theta_{n;k}(t,q)$ %and
%$\Theta_{\infty}(z,t,q)=\sum_{k\ge 0}\theta_{\infty;k}(t,q)$
satisfies 
\[
\Theta_{n}(z,t,q)=\prod_{m=1}^n\Big(1+\frac{a_m q^m z}{1-a_m q^m z t}\Big)%,\quad 
%\Theta_{\infty}(z,t,q)=\prod_{m=1}^\infty\Big(1+\frac{a_m q^m z}{1-a_m q^m z t}\Big)
.
\] It is then possible to study the distribution of $p$, or the joint distribution of $p$ and $\sigma$
on $\M_{n,k}$. 

\begin{example}[Partitions - distributions]
Setting $t=1$ we the generating function $\Theta_{n}(z,1,q)=\sum_{\vec{\ell}\in\M_{n}}w(\vec{\ell})q^{p(\vec{\ell})}z^{|\vec{\ell}|}$
simplifies to $\prod_{m=1}^n\Big(\frac{1}{1-a_m q^m z}\Big)$.
\end{example}

\begin{example}[Partitions and infinite multisets]
Moreover, setting $a_m=z=t=1$ and letting $n\to\infty$ in $\Theta_{n}(z,t,q)$ leads us directly
to the generating function of the partition function:
\[
\Theta_{\infty}(1,1,q)=\sum_{\vec{\ell}\in\M_\infty}q^{p(\vec{\ell})}=\prod_{m=1}^{\infty}\frac{1}{1-q^m}.
\]
\end{example}

\section{Distributions and limit laws}
Before, we treated the variable $t$ as a parameter, mainly between zero and one, interpolating
between weighted $k$-subsets and $k$-multisets or truncated multiple zeta values and their star counterparts. Here, we are interested in the distribution of $\sigma$ and introduce a random variable $S_{n,k}$. In the following we denote with $[z^j]$ the extraction of coefficients operator, $[z^j]f(x)=[z^j]\sum_{k=0}^{\infty}f_k \cdot z^k=f_j$, $j\in\N_0$.

\begin{defi}[Random variable $S_{n,k}$]
Given the weight sequence $\boldsymbol{a}=(a_{j})_{j\in \N}$ with positive real weights $a_j$.
The random variable $S_{n,k}$ counts the number of equal signs in an element $\vec{\ell}\in\M_{n,k}$~\eqref{eq:Mnk}.
Its probability mass function is defined in terms of $\theta_{n;k}(t)$~\eqref{def:theta}:
\[
\P\{S_{n,k}=j\}:=\frac{[t^j]\theta_{n;k}(t)}{\theta_{n;k}(1)},\quad 0\le j\le k-1;
\]
equivalently, the probability generating function $\E(t^{S_{n,k}})$
is given by 
\[
\E(t^{S_{n,k}})=\frac{\theta_{n;k}(t)}{\theta_{n;k}(1)}.
\]
\end{defi}
Note the the actual support of $S_{n,k}$ is the range $\max\{0,k-n\}\le j \le k-1$, 
since for $k \ge n$ at most $n$ elements out of a total of $k$ can be
unequal.

\smallskip

The expected value 
\[
\E(S_{n,k})=\frac{\sum_{n \ge \ell_1\ge \cdots\ge \ell_k\ge 1}\sigma(\vec{\ell})w(\vec{\ell})}{\sum_{n \ge \ell_1\ge \cdots\ge \ell_k\ge 1}w(\vec{\ell})}
\]
can be obtained by extraction of coefficients:
\begin{equation}
\label{E}
\E(S_{n,k})=\frac{\theta_{n;k}'(1)}{\theta_{n;k}(1)}=\frac{[z^k]E_t\frac{\partial}{\partial t}\Theta_{n}(z,t)}{\theta_{n;k}(1)}.
\end{equation}
Here $E_t$ denotes the evaluation operator at $t=1$. 
The variance is determined via the second factorial moment $\E(\fallfak{S_{n,k}}2)$,
\begin{equation}
\label{V}
\V(S_{n,k})=\E(\fallfak{S_{n,k}}2)+\E(S_{n,k})-\E(S_{n,k})^2,
\end{equation}
with
\[
\E(\fallfak{S_{n,k}}2)=\frac{\theta_{n;k}''(1)}{\theta_{n;k}(1)}=\frac{[z^k]E_t\frac{\partial^2}{\partial t^2}\Theta_{n}(z,t)}{\theta_{n;k}(1)}.
\]

\subsection{Interpolated \texorpdfstring{$k$}{k}-element multisets}
We choose the weight sequence $\boldsymbol{a}=(1)_{j\in\N}$ and 
obtain the distribution of $S_{n,k}$
\begin{thm}[Distribution - $k$-multisets]
The probability mass function of the random variable $S_{n,k}$ in random $k$ multisets of $[n]$, 
is given by
\[
\P\{S_{n,k}=j\}=\frac{\binom{n}{k-j}\binom{k-1}{j}}{\binom{n+k-1}k},\quad 0\le j\le k-1.
\]
Thus, $S_{n,k}\law \Hy(n+k-1,k-1,k)$ follows a hypergeometric distribution. Its expected value and variance are given by
\[
\E(S_{n,k})=\frac{k(k-1)}{n+k-1},\quad\V(S_{n,k})=\frac{k(k-1)n(n-1)}{(n+k-1)^2(n+k-2)}.
\]
Moreover, the factorial moments are given by 
\[
\E(\fallfak{S_{n,k}}{s})=\frac{\fallfak{k}{s}\cdot\fallfak{(k-1)}{s}}{\fallfak{(n+k-1)}{s}}, \quad s\ge 1.
\]
\end{thm}

\begin{proof}
The probability mass function is readily obtained by extraction of coefficients. First, we note that
\[
\Theta_{n}(z,t)=\prod_{m=1}^n\Big(1+\frac{z}{1- z t}\Big)=\Big(\frac{1-z t+z}{1-z t}\Big)^n.
\]
Then,
\begin{equation*}
\begin{split}
\P\{S_{n,k}=j\}&=\frac{[t^j]\theta_{n,k}(t)}{\theta_{n,k}(1)}=\frac{[z^k t^j]\Theta_{n,k}(z,t)}{\binom{n+k-1}{k}}
%=[z^k t^j]=\prod_{m=1}^n\Big(1+\frac{z}{1- z t}\Big). 
%=[z^k t^j]=\prod_{m=1}^n\Big(\frac{1-z(t-1)}{1-zt}\Big). 
%=[z^k t^j]=\Big(\frac{1-zt+z}{1-z t}\Big)^n
=\frac1{\binom{n+k-1}{k}}[z^{k-j} (zt)^j]\Big(\frac{1-zt+z}{1-z t}\Big)^n\\
%=\binom{n}{k-j}[(zt)^j]\frac{(1-zt)^{n-(k-j)}}{(1-zt)^{n-(n-(k-j))}}
%=\binom{n}{k-j}[(zt)^j]\frac1{(1-zt)^{n-(n-(k-j))}}
&=\frac{\binom{n}{k-j}}{\binom{n+k-1}{k}}[(zt)^j]\frac1{(1-zt)^{k-j}}=\frac{\binom{n}{k-j}\cdot \binom{k-1}{j}}{\binom{n+k-1}{k}}.
\end{split}
\end{equation*}
Alternatively, observe that $j$ equal signs are distributed amongst $k-1$ places, leading to $\binom{k-1}{j}$,
and that the remaining $k-j$ sums give a factor
\[
\sum_{n \ge \ell_1> \cdots>\ell_{k-j}\ge 1} 1 = |\S_{n,k-j}|=\binom{n}{k-j}.
\]
The expected value, the variance and the factorial moments of the hypergeometric distribution are well known, or directly obtained 
using the Vandermonde identity. 
\end{proof}

\begin{thm}[Limit laws - \texorpdfstring{$k$}{k}-multisets]
The limit laws for $\max\{n,k\}\to\infty$ are given by three different distributions:
\begin{enumerate}
	\item Degenerate case: for $n\to\infty$ and $k=o(\sqrt{n})$: $S_{n,k}\to 0$; %including $k$ fixed,
	similarly, for $k\to\infty$ and $n=o(\sqrt{k})$: $S_{n,k}-k+n\to 0$.  %including $n$ fixed,
	\item Poisson range: for $k\to\infty$ and $n\sim c \sqrt{k}$: $S_{n,k}-k+n\to \Po(c^2)$;
	similarly, for $k,n\to\infty$ and $k\sim c\cdot \sqrt{n}$ with $c>0$: $S_{n,k}\to \Po(c^2)$.
	\item Normal range: for $k,n\to\infty$ and $\sqrt{n}\ll k\ll n^2$:
	\[
	\frac{S_{n,k}-\E(S_{n,k})}{\sqrt{\V(S_{n,k})}}\to \mathcal{N}(0,1).
	\]
\end{enumerate}
\end{thm}

\begin{proof}
First we note that the expected value of the random variable $S_{n,k}$, counting the number of equal signs in random $k$ multisets of $[n]$, 
has the following behavior:
\[
\E(S_{n,k})\sim
\begin{cases}
\frac{k^2}n, \quad n\to\infty \text{ and }k=o(\sqrt{n}),\\
c, \quad k,n\to\infty \text{ and }k\sim c\cdot \sqrt{n},\\
k-n+c^2,\quad k\to\infty \text{ and }n\sim c \sqrt{k},\\
k-n+\frac{n^2}{k},\quad k\to\infty \text{ and }n=o(\sqrt{k}),\\
\end{cases}
\]
The Poisson limit law for $k\sim c\cdot \sqrt{n}$ follows directly from the asymptotic expansion 
of the factorial moments:
\[
\E(\fallfak{S_{n,k}}{s})=\frac{\fallfak{k}{s}\cdot\fallfak{(k-1)}{s}}{\fallfak{(n+k-1)}{s}}
\sim (c^2)^s.
\]
Thus, by the method of moments the (factorial) moments $S_{n,k}$ converge to the (factorial) moments of
a Poisson distributed random variable with parameter $\lambda=c^2$. In the range $k\to\infty$ and $n\sim c \sqrt{k}$
the Poisson limit law is indicated by the asymptotic expansions of $\E(S_{n,k})$ and $\V(S_{n,k})$. 
We obtain elementarily 
\[
\P\{S_{n,k}=k-n+\ell\}=\frac{\binom{n}{n-\ell}\binom{k-1}{k-n+\ell}}{\binom{n+k-1}k}.
\]
By our assumption $n\sim c \sqrt{k}$ so $\binom{n}{n-\ell}\sim (nc)^{\ell}/\ell!$.
We use the asymptotic expansion of the factorials, 
\[
n!\sim \frac{n^n}{e^n}\sqrt{2\pi n}
\]
as well as a precise expansion of the terms
\[
(n+k-1)^{n+k-1}, \quad (k-n+\ell)^{k-n+\ell},
\]
using the $\exp-\log$ representation. This gives the desired result
\[
\P\{S_{n,k}=k-n+\ell\}\sim \frac{n^{2\ell} (c^2)^{\ell}}{\ell! k^{\ell}}e^{-c^2},\quad \ell\ge 0.
\]
The normal limit laws for the hypergeometric distribution are classical and well known, see Nicholson~\cite{N} or Feller~\cite{F} and the discussion in the introduction; the proofs are omitted.
\end{proof}

\subsection{Interpolated truncated multiple zeta values}
We consider the  weight sequence $\boldsymbol{a}=(1/j)_{j\in\N}$ and 
the distribution of $S_{n,k}$, defined by the probability generating function
\[
\E(t^{S_{n,k}})=\frac{\theta_{n,k}(t)}{\theta_{n,k}(1)}
=\frac{\ztt_n(\{1\}_k)}{\zts_n(\{1\}_k)}.
\]
For example, the boundary values are given by
\[
\P\{S_{n,k}=0\}=\frac{\zt_n(\{1\}_k)}{\zts_n(\{1\}_k)},\quad \P\{S_{n,k}=k-1\}=\frac{H_n^{(k)}}{\zts_n(\{1\}_k)}.
\]
We derive the expected value
\[
\E(S_{n,k})=\frac{1}{\zts_n(\{1\}_k)} \sum_{n \ge \ell_1\ge \cdots\ge \ell_k\ge 1}\frac{\sigma(\vec{\ell})}{\ell_1^{m}\cdots \ell_k^{m}}
\]
using~\eqref{E}.

\begin{thm}
The expected value is given by the following exact expression:
\[
\E(S_{n,k})=\frac{1}{\zts_n(\{1\}_k)}\sum_{\ell=0}^{k-2}H_n^{(\ell+2)}\zts_n(\{1\}_{k-\ell-2}).
\]
\end{thm}
Note that the variance can be obtained using the second factorial moment~\eqref{V}, but it is more involved.
The limit laws for $S_{n,k}$ are technically much more involved compared to the case $\boldsymbol{a}=(1)_{j\in\N}$.
The cases of either $n$ fixed and $k\to\infty$ or $n\to\infty$ and $k$ fixed are not too difficult to analyze. In contrast, for both $n,k\to\infty$ the asymptotic expansions of the denominator $\zts_n(\{1\}_k)$ require a more precise analytic combinatorial analysis; see the works of Hwang~\cite{Hwang1994,Hwang1998,Hwang2004} and Bai et al.~\cite{Bai}. 
The case of $n$ and $k\to\infty$ indicates that a normal limit should appear. Note that the limit law, given by the sum of independent Bernoulli random variables,  also appeared in the analysis of algorithms literature; amongst others, in the the unsuccessful search in binary search trees~\cite{DobSmy1996}, the degree of the root in increasing trees~\cite{DobSmy1996},
distances in increasing trees~\cite{Dob1996,DobSmy1996} or edge-weights in increasing trees~\cite{KP2008,KW}.
\begin{prop}
For $k\to\infty$ and $n$ fixed, 
\[
k-S_{n,k}\claw D_n=B_1\oplus B_2 \oplus \dots \oplus B_n,
\]
where $B_j=\Be(\frac{1}j)$ denote independent Bernoulli-distributed random variables; hence
the probability mass function of $D$ is given by
\[
\P\{D=\ell\}=\frac{\zt_{n-1}(\{1\}_{\ell-1})}{n},\quad 1\le \ell\le n.
\]

\smallskip 

For $n\to\infty$ and $k$ fixed the random variable $S_{n,k}$ degenerates: $S_{n,k} \claw 0$.
\end{prop} 

\begin{remark}
The growth range of the degenerate case can be extended and made more precise using the results of Hwang~\cite{Hwang1995} for $\zt_n(\{1\}_k)$, stated in terms of the unsigned Stirling numbers of the first kind, and his results~\cite{Hwang2004} (see also Bai et al.~\cite{Bai}) for $\zts_n(\{1\}_k)$, related to the expected number of maxima in hypercubes. 
\end{remark}

\begin{remark}
It is well known that the sum of the limit law of $D_n$, normalized and centralized, for $n\to\infty$ is asymptotically normal:
\[
\frac{D_n-\log n}{\sqrt{\log n}}\claw \mathcal{N}(0,1).
\]
See for example Dobrow and Smythe~\cite{DobSmy1996} for an approach using Poisson approximation or Hwang~\cite{Hwang1999}. Thus, we expect that 
the random variable $S_{n,k}$ is asymptotically normal at least in a certain range of $n,k$ both tending to infinity.
We observe that $k-S_{n,k}~\sim D_n$, or $S_{n,k}\sim k-D_n$, where $D_n$ is concentrated around $\log n$. Thus $k-\log n$ should govern the behavior of $S_{n,k}$ and also its total mass $\zts_n(\{1\}_k)$.
Note further, that even the asymptotic expansions of $\zts_n(\{1\}_k)$ for $k,n\to\infty$ derived earlier by Hwang~\cite{Hwang1994,Hwang1998,Hwang2004} are more involved.
\end{remark}

\begin{proof}[Proof. (Expected value)]
Our starting point is the generating function
\[
\Theta_n(z,t)=\prod_{m=1}^n\Big(1+\frac{\frac{z}m}{1-\frac{z t}m}\Big).
\]
Differentiating with respect to $t$ gives
\begin{equation*}
\begin{split}
\frac{\partial}{\partial t}\Theta_n(z,t)
%=\bigg(\prod_{m=1}^n\Big(1+\frac{\frac{z}m}{1-\frac{z t}m}\Big)\bigg)\cdot\sum_{j=1}^{n}\frac{z^2}{j^2(1-\frac{zt}{j})^2}\cdot
%\frac{1}{1+\frac{\frac{z}j}{1-\frac{z t}j}}.
&=\Theta_n(z,t)\cdot\sum_{j=1}^{n}\frac{z^2}{j^2(1-\frac{zt}{j})^2}\cdot
\frac{1}{1+\frac{\frac{z}j}{1-\frac{z t}j}}\\
&=\Theta_n(z,t)\cdot \sum_{j=1}^{n}\frac{z^2}{j^2(1-\frac{zt}{j})(1-\frac{zt}{j}+\frac{z}{j})}.
\end{split}
\end{equation*}
Evaluation at $t=1$ gives
\[
E_t\frac{\partial}{\partial t}\Theta_n(z,t)=
\Theta_n(z,1)\cdot \sum_{j=1}^{n}\frac{z^2}{j^2(1-\frac{z}{j})}.
\]
Extraction of coefficients leads to the stated result:
\[
\E(S_{n,k})=\frac{[z^k]E_t\frac{\partial}{\partial t}\Theta_n(z,t)}{[z^k]\Theta_n(z,1)}
=\frac{\sum_{j=1}^{n}[z^{k-2}]\frac{\Theta_n(z,1)}{1-\frac{z}j}}{\zts_n(\{1\}_k)},
\]
noting that $[z^\ell]\Theta_n(z,1)=\zts_n(\{1\}_\ell)$.
\end{proof}

\begin{proof}[Proof. (Limit laws)]
For $n\to\infty$ and $k$ fixed we study 
\[
\P\{S_{n,k}=0\}=\frac{\zt_n(\{1\}_k)}{\zts_n(\{1\}_k)}.
\]
We use the representation given in Corollary~\ref{CorollZeta} for $t=0$ and $t=1$. 
We require the asymptotic expansion of the harmonic numbers $H_n$ for $n\to\infty$:
\[
H_{n}=\log n +\gamma
        +\frac{1}{2n}-\frac{1}{12n^{2}}+\mathcal{O}\Bigl(\frac{1}{n^{4}}\Bigr).
\]
Asymptotically, the summand $m_1=k$ dominates and gives for both $t=0$ and $t=1$
the same value $\frac{1}{k!}\log^k_n$. Consequently, $\P\{S_{n,k}=0\}\sim \frac{\log^k n}{\log^k n}=1$ for large $n$.

\smallskip 

For $k\to\infty$ and $n$ fixed we use singularity analysis~\cite{FS} and derive the asymptotic equivalent of
$\P\{S_{n,k}=k-j\}$, $1\le j\le n$. We have 
\[
\P\{S_{n,k}=k-j\}=\frac{1}{\zts_n(\{1\}_k)}[z^{k}t^{k-j}]\Theta_n(z,t)=\frac{1}{\zts_n(\{1\}_k)}[z^{k}t^{k-j}]\prod_{m=1}^{n}\Big(1+\frac{\frac{z}m}{1-\frac{z t}m}\Big).
\]
Since $z^{k}t^{k-j}=u^{k-j}z^j$ for $u=z t$ we get
\[
\P\{S_{n,k}=k-j\}=\frac{1}{\zts_n(\{1\}_k)}[z^{j}u^{k-j}]\prod_{m=1}^{n}\Big(1+\frac{\frac{z}m}{1-\frac{u}m}\Big).
\]
The product has a dominant singularity is at $u=1$ and can be written as 
\[
\prod_{m=1}^{n}\Big(1+\frac{\frac{z}m}{1-\frac{u}m}\Big)=\frac1{1-u}\cdot R_n(u,z),
%R_n(u,z)=\frac{\prod_{m=1}^{n}(1-\frac{u}m+\frac{ut}m)}{\prod_{m=2}^{n}(1-\frac{u}m)}
\]
with $R_n(u,z)$ analytic inside a circle of radius $2$. Consequently, for $k\to\infty$
\[
[u^{k-j}]\prod_{m=1}^{n}\Big(1+\frac{\frac{z}m}{1-\frac{u}m}\Big)\sim R_n(1,z)= n\cdot \prod_{m=1}^{n}\big(1+\frac{z-1}{m}\big).
\]
It remains to derive the asymptotic expansion of $\zts_n(\{1\}_k)$. We can use again singularity analysis as before; alternatively, 
the well known~\cite{FS,Bai,KP} binomial sum representation directly gives us the desired result:
\[
\zts_n(\{1\}_k)=\sum_{\ell=1}^{n}\binom{n}{\ell}\frac{(-1)^{\ell+1}}{\ell^k}\sim n.
\]
Finally, combining our results gives 
\[
\P\{S_{n,k}=k-j\}\sim[z^{j}]\prod_{m=1}^{n}\big(1+\frac{t-1}{m}\big),\quad 1\le j\le n.
\]
The product is exactly the probability generating function of the independent Bernoulli random variables with 
success probability $\frac1m$. Extraction of coefficients directly leads to the stated result, using the representation of $\zt_n(\{1\}_k)$ given in~\cite{KP}.
\end{proof}

\subsection{Interpolated truncated multiple zeta values - only twos}
We consider the weight sequence $\boldsymbol{a}=(1/j^2)_{j\in\N}$ and the distribution of $S_{n,k}$, defined by the probability generating function
\[
\E(t^{S_{n,k}})=\frac{\theta_{n,k}(t)}{\theta_{n,k}(1)}=\frac{\ztt_n(\{2\}_k)}{\zts_n(\{2\}_k)}.
\]
For example, the boundary values are given by
\[
\P\{S_{n,k}=0\}=\frac{\zt_n(\{2\}_k)}{\zts_n(\{2\}_k)},\quad \P\{S_{n,k}=k-1\}=\frac{H_n^{(2k)}}{\zts_n(\{2\}_k)}.
\]
We obtain the following limit laws.
\begin{thm}
For $\max\{n,k\}\to\infty$ the limit laws for $S_{n,k}$ are given by three different distributions:
\begin{enumerate}
	\item For $n\to\infty$ and $k$ fixed: $S_{n,k}\claw S_{\infty,k}$, with
	\[
	\P\{S_{\infty,k}=j\}= \frac{1}{\zts(\{2\}_k)}\cdot\sum_{\substack{\mathbf{p}\in \mathcal{P}_O(k)\\ \laenge(\mathbf{p})=k-j} }\zt(2\cdot\mathbf{p}),\quad 0\le j\le k-1.
	\]
	\item For $k\to\infty$ and $n$ fixed: $k-S_{n,k}\claw D_n$. The random variable $D_n$ is 
	given by the sum of independent Bernoulli random variables:
	\[
	D_n=B_1\oplus B_2 \oplus \dots \oplus B_n,
\]
where $B_j=\Be(\frac{1}{j^2})$ denote independent Bernoulli-distributed random variables.
Moreover, $\E(\fallfak{D_n}{s})=s!\zt_{n-1}(\{2\}_s)$, $s\ge 1$.
	
	\item For $k,n\to\infty$:   $k-S_{n,k}\claw D$;
	\[
	D=B_1\oplus B_2 \oplus \dots = \bigoplus_{m=1}^{\infty}B_m.
\]
Here $B_j=\Be(\frac{1}{j^2})$, $j\ge 1$, denote independent Bernoulli-distributed random variables.
Moreover, $\E(\fallfak{D}{s})=s!\zt(\{2\}_s)$, $s\ge 1$.
\end{enumerate}
\end{thm}

\begin{remark}
The random variable $D$ is exactly the limit law of the number of cuts in a recursive tree to isolate a leaf~\cite{KP2008b}; modified weight sequences lead to other families of increasing trees. Moreover, $D_n$ is closely related the the distribution of the number of cuts in a tree of size $n$.
\end{remark}

\begin{proof}
Similar to~\eqref{EqYamamotoOrdered} we have,
\[
\ztt_{n}(\{2\}_{k})  =\sum_{\circ = \text{``},\text{''} \text{or} \, \text{``}+\text{''}}t^{\sigma}\zt_{n}( \underbrace{2 \circ 2 \circ \dots \circ 2}_{k}) = \sum_{\mathbf{p}\in \mathcal{P}_O(k)}t^{k-\laenge(\mathbf{p})}\zt_{n}(2\cdot \mathbf{p}).
\]
By taking the limit $n\to\infty$, $k$ being fixed, and extraction of coefficients we directly obtain the stated result.

\smallskip

For $k\to\infty$ and arbitrary $n$ we use singularity analysis~\cite{FS} and derive the asymptotic equivalent of
$\P\{S_{n,k}=k-j\}$, $1\le j$. We have 
\[
\P\{S_{n,k}=k-j\}=\frac{1}{\zts_n(\{2\}_k)}[z^{k}t^{k-j}]\Theta_n(z,t)=\frac{1}{\zts_n(\{2\}_k)}[z^{k}t^{k-j}]\prod_{m=1}^{n}\Big(1+\frac{\frac{z}{m^2}}{1-\frac{z t}{m^2}}\Big).
\]
Setting as before $u=z t$, we get
\[
\P\{S_{n,k}=k-j\}=\frac{1}{\zts_n(\{2\}_k)}[z^{j}u^{k-j}]\prod_{m=1}^{n}\Big(1+\frac{\frac{z}{m^2}}{1-\frac{u}{m^2}}\Big).
\]
The product has a dominant singularity is at $u=1$ and can be written as 
\[
\prod_{m=1}^{n}\Big(1+\frac{\frac{z}{m^2}}{1-\frac{u}{m^2}}\Big)=\frac1{1-u}\cdot R_n(u,z),
%R_n(u,z)=\frac{\prod_{m=1}^{n}(1-\frac{u}m+\frac{ut}m)}{\prod_{m=2}^{n}(1-\frac{u}m)}
\]
with $R_n(u,z)$ analytic inside a circle of radius $2$. Consequently, for $k\to\infty$
\[
[u^{k-j}]\prod_{m=1}^{n}\Big(1+\frac{\frac{z}{m^2}}{1-\frac{u}{m^2}}\Big)\sim R_n(1,z) = \frac{2n}{n+1}\cdot \prod_{m=1}^{n}\big(1+\frac{z-1}{m^2}\big).
\]
For the asymptotic expansion of $\zts_n(\{2\}_k)$ we can use again singularity analysis and obtain
$\zts_n(\{2\}_k)\sim \frac{2n}{n+1}$. Finally, combining our results gives 
\[
\P\{S_{n,k}=k-j\}\sim[z^{j}]\prod_{m=1}^{n}\big(1+\frac{t-1}{m^2}\big),\quad 1\le j\le n.
\]
The product is exactly the probability generating function of $n$ independent Bernoulli random variables each with 
success probability $\frac1{m^2}$ for $1\le m\le n$. 
\end{proof}

\subsection{Refined decomposition of sigma, random vectors and marginals}
Similarly to the refinements of the parameter $\sigma$ into $\sigma^{(i)}$, $1\le i\le n$, the random variable 
$S_{n,k}$ can be decomposed into $S_{n,k}^{(i)}$: $S_{n,k}=\sum_{i=1}^{n}S_{n,k}^{(i)}$.
Given the weight sequence $\boldsymbol{a}=(a_{j})_{j\in \N}$, 
the joint distribution of the random vector $\mathbf{S}_{n,k}=(S_{n,k}^{(1)},\dots, S_{n,k}^{(n)})$ is determined by
\[
\P\{\mathbf{S}_{n,k}=\vec{j}\}=[\vec{t}^{\vec{j}}]\frac{\theta_{n;k}(\vec{t})}{\theta_{n;k}(\vec{1})}.
\]

Another possibility is to distinguish only between odd and even indices,
such that $S_{n,k}=S_{n,k}^{(O)}+S_{n,k}^{(E)}$, 
and to study the bivariate probability generating function
\[
\E(t_O^{S_{n,k}^{(O)}}t_E^{S_{n,k}^{(E)}})=\frac{\theta_{n;k}(t_O,t_E)}{\theta_{n;k}(1,1)}.
\]

For truncated zeta values with the weight sequence $\boldsymbol{a}=(1)_{j\in\N}$, the random vector $\mathbf{S}_{n,k}$ has the probability generating function
\[
\E(\mathbf{t}^{\mathbf{S}_{n,k}})
=\frac{1}{\binom{n+k-1}{k}}[z^k]\prod_{m=1}^{n}\Big(1+\frac{z}{1- z t_m}\Big).
\]

\begin{example}[Truncated zeta values - random vector]
For truncated zeta values, weight sequence $\boldsymbol{a}=(1/j)_{j\in\N}$, the random vector $\mathbf{S}_{n,k}$ has the probability generating function
\[
\E(\mathbf{t}^{\mathbf{S}_{n,k}})
=\frac{1}{\zts(\{1\}_k)}[z^k]\prod_{m=1}^{n}\Big(1+\frac{\frac{z}m}{1- \frac{z t_m}m}\Big).
\]
\end{example}

\begin{example}[Truncated zeta values - marginals]
In contrast to the multiset case, the marginals $S_{n,k}^{(i)}$ are not exchangeable anymore. 
The probability generating functions are given by
\[
\E(t^{S_{n,k}^{(i)}})
=\frac{1}{\zts(\{1\}_k)}[z^k](1-\frac{z}{i})\Big(1+\frac{\frac{z}i}{1- \frac{z t}i}\Big)\cdot\prod_{m=1}^{n}\Big(\frac{1}{1-\frac{z}m}\Big).
\]
It is expected that the marginals are asymptotically independent, at least in some growth range of $k$ and $n$.
\end{example}

\smallskip

\begin{thm}[Marginals - interpolated multisets]
\label{TheMar}
The marginals $S_{n,k}^{(i)}$ are identically distributed, but not independent; the sequence $\mathbf{S}_{n,k}$ is exchangeable. 
The probability mass function is given by
\[
\P\{S_{n,k}^{(i)}=j\}=
\begin{cases}
\frac{\binom{n-2+k}{k}+\binom{n-3+k}{k}}{\binom{n+k-1}{k}}, \quad j=0,\\
\frac{\binom{n-3+k-j}{k-j-1}}{\binom{n+k-1}{k}},\quad 1\le j \le k-1.
\end{cases}
\]
The factorial moments, $s\ge 1$, are given by
\[
\E(\fallfak{S_{n,k}}{s})=s!\cdot\frac{\fallfak{k}{s+1}}{(n+k-1)\fallfak{(n+s-1)}{s}}.
\]
\end{thm}

\begin{thm}
\label{TheMarLimit}
We obtain the following limit laws, depending on the growth of $k$ and $n$:
\begin{itemize}
	\item for $k\to \infty$ and $n$ fixed the normalized random variable $\frac{S_{n,k}}{k}$ converges to a Beta-distributed random variable $\frac{S_{n,k}}{k}\claw \Beta(1,n-1)$. 
	\item for $k/n\to \infty$ and $k,n\to \infty$ the normalized random variable $\frac{n}{k}S_{n,k}$ converges to a standard exponentially distributed random variable, $\frac{n}{k}S_{n,k}\claw \Exp(1)$. 
	\item for $k/n\to c>0$ and $k,n\to \infty$ the random variable $S_{n,k}$ converges to a (modified) geometric distribution $G$:
	\[
	\P\{G=j\}=\frac{c^{j+1}}{(1+c)^{j+2}}, \quad j\ge 1\quad, \P\{G=0\}=\frac{1}{1+c}+\frac{c}{(1+c)^2}.
	\]
	\item for $k/n\to 0$ and $n\to \infty$ the random variable $S_{n,k}$ degenerates, $S_{n,k}\claw 0$.
\end{itemize}

\end{thm}

\begin{remark}
The random variable $S_{n,k}^{(i)}$ is closely related to the number of descendants $D_{n+k,k}$ of node labelled $k$ in a recursive trees of size $n+k$~\cite{KP2006}.
Thus, $S_{n,k}^{(i)}$ can be described in terms of a Polya-urn model and the asymptotics for $k\to \infty$ and $n$ fixed 
can be refined to almost sure convergence; moreover a classical result is available in terms of a martingale tail sum~\cite{HH80}.
\end{remark}

\begin{proof}[Proof of Theorem~\ref{TheMar}]
Let $1\le i\le n$. The probability generating function of the marginal distribution $S_{n,k}^{(i)}$ is given by
\[
\E(t^{S_{n,k}^{(i)}})
=\frac{1}{\binom{n+k-1}{k}}[z^k]\frac{1}{(1-z)^{n-1}}\cdot\Big(1+\frac{z}{1- z t}\Big).
\]
By the structure of the probability generating function, the random variables are identically distributed and the random vector is exchangable. Thus, we can directly obtain the expected value from $\E(S_{n,k})$.
\[
\E(S_{n,k}^{(i)})=\frac{1}n\E(S_{n,k})=\frac{k(k-1)}{n(n+k-1)}.
%=\frac{\binom{n+k-2}{k-2}}{\binom{n+k-1}{k}}
\]
Additionally, the probability mass function is obtained readily by extraction of coefficients.
All factorial moments, $s\ge 1$, can be obtained in a straightforward way:
\begin{equation*}
\begin{split}
\E(\fallfak{S_{n,k}}{s})
&=[z^k]E_t\frac{\partial^s}{\partial t^s}\E(t^{S_{n,k}^{(i)}})
=\frac{s!}{\binom{n+k-1}{k}}[z^k]\frac{z^{s+1}}{(1-z)^{n+s}}\\
&=\frac{s!\binom{n+k-2}{k-s-1}}{\binom{n+k-1}{k}}
=s!\cdot\frac{\fallfak{k}{s+1}}{(n+k-1)\fallfak{(n+s-1)}{s}}.
\end{split}
\end{equation*}
\end{proof}

\begin{proof}[Proof of Theorem~\ref{TheMarLimit}]
We use the method of moments and derive asymptotic expansions of the factorial moments. For $n$ fixed and $k\to\infty$ 
we get
\[
\E(\fallfak{S_{n,k}}{s})=s!\cdot\frac{\fallfak{k}{s+1}}{(n+k-1)\fallfak{(n+s-1)}{s}}
\sim \frac{s!}{\fallfak{(n+s-1)}{s}}\cdot k^s.
\]
Consequently, the power moments of $S_{n,k}/k$ are asymptotically given by $\frac{s!}{\fallfak{(n+s-1)}{s}}$, 
which proves the Beta limiting distributions. 

For $n\to \infty$ we observe that 
\[
\E(\fallfak{S_{n,k}}{s})\sim s!\cdot \lambda_{n,k}^s\cdot \frac{k}{n+k},\quad\text{with } \lambda_{n,k}=\frac{k}{n}.
\]
Thus, the factorial moments are almost of mixed Poisson type~\cite{KP2016} with standard exponential mixing distribution; the additional factor $\frac{k}{n+k}$ can be explained by the definition of the parameter $\sigma$, which influences the discrete limit case. This directly leads to the stated limit laws using Lemma 2 of~\cite{KP2016}. Alternatively, the discrete limit for $k/n\to c$ can be directly obtained as follows:
\[
\P\{S_{n,k}^{(i)}=j\}=\frac{\binom{n-3+k-j}{k-j-1}}{\binom{n+k-1}{k}}
=\frac{\fallfak{k}{j+1}\cdot (n-1)}{\fallfak{(n+k-1)}{j+2}}
\sim \frac{c^{j+1}}{(1+c)^{j+2}},\quad j\ge 1. 
\]
The remaining case $j=0$ is treated in a similar way. For $k/n\to 0$ we observe that
\[
\P\{S_{n,k}^{(i)}=0\}=\frac{n-1}{n+k-1}+\frac{k(n-1)}{\fallfak{(n+k-1)}{2}}\sim 1.
\]
\end{proof}

\subsection{Sum theorem for interpolated multiple zeta values}
Yamamoto~\cite{Y} established, amongst many other things, the sum theorem for interpolated multiple zeta values:
\begin{equation}
\label{TheYam}
\sum_{\substack{k_1\ge 2, k_i\ge 1\\ \sum_{\ell=1}^{n}k_\ell =k}}\ztt(k_1,\dots,k_n)
=\zt(k)\cdot \sum_{j=0}^{n-1}\binom{k-1}j t^j(1-t)^{n-1-j},
\end{equation}
with $k>n$. Note that the case $t=1$ gives the sum theorem for multiple zeta star values 
\[
\sum_{\substack{k_1\ge 2, k_i\ge 1\\ \sum_{\ell=1}^{n}k_\ell =k}}\zts(k_1,\dots,k_n)
=\zt(k)\cdot \binom{k-1}{n-1},
\]
whereas $t=0$ gives the ordinary sum theorem. 

\smallskip

Using~\eqref{TheYam}, we can study the distribution of the parameter $\sigma$ on sums of interpolated multiple zeta values with the same depth $n$ and weight $k$.
Let $\mathcal{S}_{n,k}$ denote the random variable with probability generating function 
\[
\E(t^{\mathcal{S}_{n,k}})=\frac{\zt(k)\cdot \sum_{j=0}^{n-1}\binom{k-1}j t^j(1-t)^{n-1-j}}{\zt(k)\cdot \binom{k-1}{n-1}}
=\frac{\sum_{j=0}^{n-1}\binom{k-1}j t^j(1-t)^{n-1-j}}{\binom{k-1}{n-1}}.
\]
Let $B_{j,n}(t)$ denote the Bernstein polynomials
\[
B_{j,n}(t)=\binom{n}j t^j (1-t)^{n-j},\quad 0\le j\le n. 
\]
We obtain the following result.
\begin{thm}
The probability generating function $\E(t^{\mathcal{S}_{n,k}})$ is given by a Bernstein form of degree $n-1$ with Bezier coefficients
$\beta_{j}=\frac{1}{\binom{k-1-j}{k-n}}$
\[
\E(t^{\mathcal{S}_{n,k}})%=\frac1{\binom{k-1}{n-1}}\sum_{j=0}^{n-1}\frac{\binom{k-1}j}{\binom{n-1}j}B_{j,n-1}(t)
=\sum_{j=0}^{n-1}\beta_j\cdot B_{j,n-1}(t).
\]
Let $\mathcal{R}_{n,k}=n-\mathcal{S}_{n,k}$. The probability mass function $\P\{\mathcal{R}_{n,k}=i\}$, is given by
%\[
%\P\{\mathcal{R}_{n,k}=i\}=\frac{\binom{k-1-i}{k-n-1}}{\binom{k-1}{k-n}},\quad 1\le i\le n,
%\]
\[
\P\{\mathcal{R}_{n,k}=i\}=\frac{\binom{k-1-i}{j-2}}{\binom{k-1}{j-1}},\quad 1\le i\le n,
\]
setting with $n=k-j+1$ and $2\le j\le k$. Consequently, $\mathcal{R}_{n,k}$, with $n=k-j+1$, has the same distribution as the random variable $D_{n,j}$, counting the number of descendants of node $j$ in a random recursive tree of size $n$:
$\mathcal{R}_{n,k}\law D_{n,j}$.
\end{thm}

As a byproduct of our identification of $\mathcal{R}_{n,k}$, we get the following limit laws from~\cite{KP2006}.
\begin{coroll}
The limiting distribution behaviour of the random variable $\mathcal{R}_{n,k}=n-\mathcal{S}_{n,k}$, with $n=k-j+1$ and $2\le j\le k$, 
is, for $k \to \infty$ and depending on the growth of $j=j(k)$, characterized as follows.
\begin{itemize}
\item for $j$ fixed, $\frac{\mathcal{R}_{n,k}}{k} \xrightarrow{(d)} \beta(1,j-1)$

\item for small $j$: $j \to \infty$ such that $j = o(k)$: the normalized random variable $\frac{j}{n}\mathcal{R}_{n,k}$ is asymptotically Exponential distributed,

\item for $j$: $j \to \infty$ such that $j \sim \rho k$, with $0 < \rho < 1$.
The shifted random variable $\mathcal{R}_{n,k}-1$ is asymptotically negative binomial-distributed, 
$\mathcal{R}_{n,k} -1 \xrightarrow{(d)} \text{NegBin}(1,\rho)$,

\item for large $j$: $j \to \infty$ such that $\ell := n-j = o(n)$: $\mathcal{R}_{n,k}\to 1$, i.e. $\lim_{k\to\infty}\mathbb{P}\{\mathcal{R}_{n,k} = 1\} = 1$.
\end{itemize}
\end{coroll}

\begin{proof}
We note that by definition, the summands are weighted Bernstein polynomials:
\[
\E(t^{\mathcal{S}_{n,k}})=\frac1{\binom{k-1}{n-1}}\sum_{j=0}^{n-1}\frac{\binom{k-1}j}{\binom{n-1}j}B_{j,n-1}(t).
\]
Simplification gives the stated expression for the Bezier coefficients. 
The probability mass function is obtained by extraction of coefficients:
\begin{align*}
\P\{\mathcal{S}_{n,k}=i\}&=[t^i]\E(t^{\mathcal{S}_{n,k}})
=\frac1{\binom{k-1}{n-1}}\sum_{j=0}^{i}\binom{k-1}j[t^{i-j}](1-t)^{n-1-j}\\
&=\frac1{\binom{k-1}{n-1}}\sum_{j=0}^{i}\binom{k-1}j\binom{n-1-j}{i-j}(-1)^{i-j}.
\end{align*}
Using an identity for binomial coefficients~\cite{GraKnuPa} then readily gives
\[
\P\{\mathcal{S}_{n,k}=i\}=\frac{\binom{k-n+i-1}{k-n-1}}{\binom{k-1}{n-1}},\quad 0\le i\le n-1,
\]
and thus we get the corresponding result for $\mathcal{R}_{n,k}=n-\mathcal{S}_{n,k}$. We observe that for $n=k-j+1$ the probability mass functions and thus the distribution of $\mathcal{R}_{n,k}$
and the random variable $D_{n,j}$ (case $c_2=0$) in~\cite{KP2006} coincide and thus we can directly transfer the limit laws for $D_{n,j}$.
\end{proof}

\section{Summary and Outlook}
We introduced a parameter $\sigma$ on weighted $k$-element multisets and studied properties of it using symbolic combinatorics. 
This allows to prove several relations for truncated interpolated multiple zeta values $\ztt_n(\{m\}_k)$ ( 
as well as reproving identities for truncated multiple zeta values $\ztt(\{m\}_k)$). Introducing refined enumeration leads
to new refinements of previous identities. Interpreting the parameter $\sigma$ as a random variable $S_{n,k}$
leads to several different limit laws, depending on the considered weight sequences and the growth of $n$ and $k$. Interestingly, the limits laws are closely related to a great many results in combinatorial probability and analytic combinatorics.

It is of interest to complete the analysis of $S_{n,k}$ in the case $\boldsymbol{a}=(1/j)$; we will report on this elsewhere. 
Similar to the random variable $S_{n,k}$, defined in terms of the parameter $\sigma_{n,k}$, one may also study the distribution of the parameter $p=p_{n,k}$ or the joint distribution of $p$ and $\sigma$. 
Moreover, it is certainly of interest to study the distribution of $S_{n,k}$ for other interesting sequences, compare with Vignat and Wakhare~\cite{VW} or Hoffman and Mez\H{o}~\cite{HM}.

\subsection{Acknowledgement}
The author warmly thanks Paul Schreivogl for pointing out the connection to the Bernstein polynomials. Moreover,
the author thanks the referees for their great help and very valuable suggestions, improving significantly the structure and presentation of this work, clarifying a few points, as well as providing additional references.

\section{Appendix: Auxiliary results about probability distributions\label{SubProb}}
{\footnotesize
A beta-distributed random variable $Z \law \beta(\alpha,\beta)$ with parameters $\alpha,\beta>0$ has 
a probability density function given by $f(x)=\frac1{B(\alpha,\beta)}x^{\alpha-1}(1-x)^{\beta-1}$, where 
$B(\alpha,\beta)=\frac{\Gamma(\alpha)\Gamma(\beta)}{\Gamma(\alpha+\beta)}$ denotes the Beta-function.
The (power) moments of $Z$ are given by 
\[
\E(Z^s)=\frac{\prod_{j=0}^{s-1}(\alpha+j)}{\prod_{j=0}^{s-1}(\alpha+\beta+j)}
=\frac{\fallfak{(\alpha+s-1)}{s}}{\fallfak{(\alpha+\beta+s-1)}{s}},\quad s\ge 1.
\]
The beta-distribution is uniquely determined by the sequence of its moments. In this work we will discuss
a beta-distributed random variable $Z \law \beta(1,n-1)$, with moments given by
$\E(Z^s)=\frac{s!}{\fallfak{(n+s-1)}{s}}$, for $s\ge 1$.

\smallskip

An exponentially distributed random variable $Z\law\Exp(1)$ with parameter one has density $f(x)=e^{-x}$, $x\ge 0$ and power moments $\E(Z^s)=s!$.

\smallskip

A Bernoulli distributed random variable $B\law\Be(p)$ with parameter $p\in[0,1]$ has probability mass function determined by
$\P\{B=1\}=p$, $\P\{B=0\}=1-p$. The sum $Z_n=\bigoplus_{j=1}^{n}B_j$ of $n$ independent Bernoulli distributed random variables $B_j=\Be(p_j)$ 
has probability generating function
\[
\E(v^{Z_n})=\prod_{j=1}^{n}\E(v^{B_j})=\prod_{j=1}^{n}\big(1+(v-1)p_j\big).
\]
The factorial moments $\E(\fallfak{Z_n}{s})$ of $Z_n$ are given by
\[
\E(\fallfak{Z_n}{s})=E_v\frac{\partial^s}{\partial v^s}\E(v^{Z_n})
= s!\cdot\sum_{n \ge \ell_1> \cdots> \ell_s\ge 1}p_{\ell_1}p_{\ell_2}\dots p_{\ell_s}.
\]
Under the assumption that the probabilities $p_j$ tend to zero fast enough, 
we may define $Z=Z_\infty$ as the sum of infinitely many Bernoulli random variable:
$Z=\bigoplus_{j=1}^{\infty}B_j$, 
with factorial moments formally given by $s!$ times a multiple series:
\[
\E(\fallfak{Z}{s})= s!\cdot\sum_{\ell_1> \cdots> \ell_s\ge 1}p_{\ell_1}p_{\ell_2}\dots p_{\ell_s}.
\]

\smallskip

A Poisson distributed random variable $Z\law \Po(\lambda)$, $\lambda >0$, has probability mass function and factorial moments given by
\[
\P\{Z=j\}=\frac{\lambda^j}{j!}\cdot e^{\lambda},\,j\ge 0,\quad \E(\fallfak{Z}{s})=\lambda^s,\,s\ge 1.
\]

\smallskip 

A hypergeometric distributed random variable $Z\law\Hy(N,K,n)$ with parameters $N,K,n$ has probability mass function and factorial moments given by
\[
\P\{Z=j\}=\frac{\binom{K}{j}\binom{N-K}{n-j}}{\binom{N}{n}},0\le j\le n\quad 
\E(\fallfak{Z}{s})=\frac{\fallfak{K}{s}\fallfak{n}{s}}{\fallfak{N}{s}},\,s\ge 1.
\]
In particular, 
\[
\E(Z)=n\cdot\frac{ K}{N},\quad \V(Z)=n\cdot \frac{K(N-K)(N-n)}{N^2(N-1)}.
\]
%\begin{equation*}
%\begin{split}
%\E(\fallfak{Z}{s})&
%=\sum_{j=0}^{n}\fallfak{j}{s}\frac{\binom{K}{j}\binom{N-K}{n-j}}{\binom{N}{n}}\\
%&=\sum_{j=s}^{n}\fallfak{j}{s}\frac{\binom{K}{j}\binom{N-K}{n-j}}{\binom{N}{n}}\\
%&=\sum_{j=s}^{n}\fallfak{j}{s}\frac{\frac{K!}{j!(K-j)!}}{\binom{N-K}{n-j}}{\binom{N}{n}}\\
%&=\sum_{j=s}^{n}\fallfak{K}{s}\frac{\binom{K-s}{j-s}\binom{N-K}{n-j}}{\binom{N}{n}}\\
%&=\sum_{j=0}^{n-s}\fallfak{K}{s}\frac{\binom{K-s}{j}\binom{N-K}{n-j-s}}{\binom{N}{n}}\\
%&=\fallfak{K}{s}\frac{\binom{N-s}{n-s}}{\binom{N}{n}}\\
%&=\fallfak{K}{s}\frac{\fallfak{n}{s}}{\fallfak{N}{s}}.
%\end{split}
%\end{equation*}
Moreover, the following normal limit law can be deduced (see Nicholson~\cite{N} or Feller~\cite{F}).
\begin{lem}
Let $Z\law\Hy(N,K,n)$ denote a hypergeometric distributed random variable. 
Under the assumption $\min\{K,N\}\to\infty$ and $n$ such that $\E(Z),\V(Z)\to\infty$,  
the random variable $Z$, centered and normalized, is asymptotically standard normal distributed:
\[
\frac{Z-\E(Z)}{\sqrt{\V(Z)}}\claw \mathcal{N}(0,1).
\]
\end{lem}

}

\end{document}